\documentclass[twoside]{article}
\usepackage[utf8]{inputenc}
\usepackage[T2A]{fontenc}
\usepackage{array}
\usepackage{amssymb}
\usepackage{amsmath}
\usepackage{amsthm}
\usepackage{latexsym}
\usepackage{hyperref}
\usepackage{indentfirst}
\usepackage{appendix}
\usepackage{bm}
\usepackage{enumerate}
\usepackage[dvips]{graphicx}
\usepackage{epsf}
\usepackage{wrapfig}
\usepackage[mathscr]{euscript}
\usepackage{indentfirst}
\usepackage[english,russian]{babel}
\usepackage{textcomp}
\usepackage{txfonts}
\usepackage{diagrams}
\newtheorem{theorem}{Theorem}
\newtheorem{lemma}{Lemma}
\newtheorem{definition}{Definition}
\newtheorem{proposition}{Proposition}
\newtheorem{example}{Example}
\newtheorem{corollary}{Corollary}
\DeclareMathSymbol{\boxRight}{\mathrel}{symbolsC}{136} 
\DeclareMathSymbol{\diamondRight}{\mathrel}{symbolsC}{140} 
\DeclareMathSymbol{\boxto}{\mathrel}{symbolsC}{128}
\DeclareMathSymbol{\diamondto}{\mathrel}{symbolsC}{132}
\newcommand*\vDdash{%
	\mathrel{%
		\ooalign{$\vdash$\cr$\vDash$}%
	}%
}
\def\@clipped@vdash{%
	\raise .6ex\hbox{\clipbox{0pt .6ex 0pt .6ex}{$\vdash$}}%
}

\begin{document}
{\selectlanguage{english}
\binoppenalty = 10000 %
\relpenalty   = 10000 %

\pagestyle{headings} \makeatletter
\renewcommand{\@evenhead}{\raisebox{0pt}[\headheight][0pt]{\vbox{\hbox to\textwidth{\thepage\hfill \strut {\small Grigory K. Olkhovikov}}\hrule}}}
\renewcommand{\@oddhead}{\raisebox{0pt}[\headheight][0pt]{\vbox{\hbox to\textwidth{{A complete basic intuitionistic conditional logic}\hfill \strut\thepage}\hrule}}}
\makeatother

\title{An intuitionistically complete system of basic intuitionistic conditional logic}
\author{Grigory K. Olkhovikov\\ Department of Philosophy I\\ Ruhr University Bochum\\
email: grigory.olkhovikov@\{rub.de, gmail.com\}}
\date{}
\maketitle
\begin{quote}
{\bf Abstract.} We introduce a basic intuitionistic conditional logic $\mathsf{IntCK}$ that we show to be complete both relative to a special type of Kripke models and relative to a standard translation into first-order intuitionistic logic. We show that $\mathsf{IntCK}$ stands in a very natural relation to other similar logics, like the basic classical conditional logic $\mathsf{CK}$ and the basic intuitionistic modal logic $\mathsf{IK}$. As for the basic intuitionistic conditional logic $\mathsf{ICK}$ proposed in \cite{weiss}, $\mathsf{IntCK}$ extends its language with a diamond-like conditional modality $\diamondto$, but its ($\diamondto$)-free fragment is also a proper extension of $\mathsf{ICK}$. We briefly discuss the resulting gap between the two candidate systems of basic intuitionistic conditional logic and the possible pros and cons of both candidates.
\end{quote}
\begin{quote}{\bf Keywords.} first-order logic, intuitionistic logic, strong completeness, conditional logic
\end{quote}

\section{Introduction}
The present paper was written in an attempt to find and vindicate an answer to the question, what is a basic intuitionistic conditional logic. By \textit{basic} logic we mean a logic that is complete relative to a universal class of suitably defined Kripke models. This basic logic must also be \textit{intuitionistic} in the sense of being the intuitionistic counterpart of the basic classical conditional logic $\mathsf{CK}$ (introduced in \cite{chellas}; see also \cite{segerberg}). More precisely, its only difference from  $\mathsf{CK}$ must consist in the fact that the classical reading of the Kripke semantics for $\mathsf{CK}$ is replaced by intuitionistic reading, and the classical first-order metalogic of $\mathsf{CK}$ is replaced with intuitionistic one. Finally, this logic must be fully \textit{conditional} in that it must enrich the language of intuitionistic propositional logic with the full set of conditional modalities, both the stronger box-like $\boxto$ and the weaker diamond-like $\diamondto$.

The existing literature on intuitionistic conditional logic is not vast, but it seems to have a clear candidate for the role of a basic system. This candidate is the system $\mathsf{ICK}$, introduced by Y. Weiss in \cite{weiss} (see also the more detailed exposition in \cite{weiss-thesis}). This research was followed by \cite{ciardelliliu}, which was based on a version of Kripke semantics different from that used in \cite{weiss}; however, this change did not affect the status of $\mathsf{ICK}$ which the authors of \cite{ciardelliliu} have shown to be also the system complete relative to the universal class of their preferred variety of Kripke models. $\mathsf{ICK}$ also has a claim for the title of the intuitionistic counterpart of $\mathsf{CK}$ in that its complete axiomatization results from replacing the classical propositional fragment of $\mathsf{CK}$ with the intuitionistic one.

However, $\mathsf{ICK}$  only features the stronger conditional modality, $\boxto$, and no attempt is made to show the completeness of $\mathsf{ICK}$ relative to an intuitionistic reading of the metatheory of $\mathsf{CK}$. The former shortcoming is openly acknowledged by I. Ciardelli and X. Liu, who write: `` Just like $\forall$ and $\exists$ are not interdefinable in intuitionistic predicate logic, and $\Box$ and $\Diamond$ are not interdefinable in intuitionistic modal
logic, also $\boxto$ and $\diamondto$ will not be interdefinable in intuitionistic conditional logic.
In order to capture might-conditionals in the intuitionistic setting, we need to add $\diamondto$ to the language as a new primitive... In this extended setting, the properties of $\boxto$ might not
uniquely determine the properties of $\diamondto$. Therefore, it becomes especially interesting to look at the landscape of intuitionistic conditional logics in a setting where the language comprises both operators''\cite[p. 830]{ciardelliliu}.

In this paper, we are going to propose a different system for the role of basic intuitionistic conditional logic. This system, which we call $\mathsf{IntCK}$, both answers the concern expressed in the above quote by I. Ciardelli and X. Liu and can be shown to enjoy a form of strong completeness relative to an intuitionistic reading of its metatheory. We will also show that the ($\diamondto$)-free fragment of $\mathsf{IntCK}$ is a proper extension of $\mathsf{ICK}$, whence it follows that $\mathsf{ICK}$ is incomplete relative to intuitionistic reading of its metatheory and hence cannot be viewed as a full basic intuitionistic conditional logic even within the ($\diamondto$)-free fragment of the conditional language.

The rest of this paper is organized as follows. Section \ref{S:Prel} introduces the notational preliminaries, followed by Section \ref{S:system} where we explain the syntax and a form of Kripke semantics of $\mathsf{IntCK}$ (in Section \ref{sub:lands}) and then axiomatize the logic (in Section \ref{sub:axiomatization}). Section \ref{S:other} explores the relation of $\mathsf{IntCK}$ to other logics, both propositional (in Section \ref{sub:intense}) and first-order (in Section \ref{sub:foil}). The latter subsection also contains the completeness theorem for $\mathsf{IntCK}$ relative to an intuitionistic reading of its Kripke semantics.

Finally, in Section \ref{S:conclusion}, we briefly discuss the results of the previous sections, and, after drawing some conclusions, describe several avenues for continuing the research lines presented in the paper. The paper also has several appendices where the reader can find the more technical parts of our reasoning which we include for the sake of completeness.

\section{Preliminaries}\label{S:Prel}
We use this section to fix some notations to be used throughout this (rather lengthy) paper.

We will use IH as the abbreviation for Induction Hypothesis in the inductive proofs, and we will write $\alpha:=\beta$ to mean that we define $\alpha$ as $\beta$. We will use the usual notations for sets and functions. As for the sets, we will write $X\Subset Y$, iff $X\subseteq Y$ and $X$ is finite. Furthermore, we will understand the natural numbers as the finite von Neumann ordinals. We denote by $\omega$ the smallest infinite ordinal; given two ordinals $\lambda$, $\mu$, we will use $\lambda \in \mu$ and $\lambda < \mu$ interchangeably. Given a (finite) tuple of any sort of objects $\alpha = (x_1,\ldots,x_n)$, we will denote by $init(\alpha)$ and $end(\alpha)$ the initial and final element of $\alpha$, that is to say, $x_1$ and $x_n$, respectively. More generally, given any $i < \omega$ such that $1 \leq i \leq n$, we set that $\pi^i(\alpha):= x_i$, i.e. that $\pi^i(\alpha)$ denotes the $i$-th projection of $\alpha$. Given another tuple $\beta = (y_1,\ldots,y_m)$, we will denote by $(\alpha)^\frown(\beta)$ the concatenation of the two tuples, i.e. the tuple $(x_1,\ldots,x_n,y_1,\ldots,y_m)$. The empty tuple will be denoted by $\Lambda$.

We will extensively use ordered couples of sets which we will also call \textit{bi-sets}. The usual set-theoretic relations and operations on bi-sets will be understood componentwise, so that, e.g. $(X,Y)\subseteq(Z,W)$ means that $X \subseteq Z$ and $Y \subseteq Y$ and similarly in other cases.

The relations will be understood as sets of ordered tuples where the length of the tuple defines the \textit{arity} of the relation. Given binary relations $R \subseteq X \times Y$ and $S\subseteq Y\times Z$, we will denote their composition by $R\circ S:= \{(a,c)\mid\text{ for some }b \in Y,\,(a, b)\in R\,\&\,(b,c)\in S\}$.

Given a set $X$, we will denote by $id[X]$ the identity function on $X$, i.e. the function $f:X\to X$ such that $f(x) = x$ for every $x \in X$. It is clear that a function $f:X\to Y$ can be understood as a special type of relation $f\subseteq X\times Y$. Therefore, our notation for the composition of functions is in line with the one used for the composition of relations: namely, given two functions $f:X\to Y$ and $g:Y\to Z$, we denote by $f\circ g$ the function $h:X\to Z$ such that $h(x) =g(f(x))$ for every $x \in X$.\footnote{It seems that, in the existing literature, the more popular option is to define such function $h$ by $g\circ f$ (see, e.g., \cite{awodey}). However, in the logical literature our convention on the composition of relations seems to be more popular than the alternative one, viz., to denote by $S\circ R$ the relation which we denote by $R\circ S$; and since the functions go hand-in-hand with relations, we decided to adopt a similar notation for the composition of functions as well. } In relation to the operation of superposition, functions of the form $id[X]$ have a special importance as a limiting case. More precisely, given a set $X$ and a family $F$ of functions from $X$ to $X$, we will also assume that the superposition of an empty tuple of functions from $F$ is just $id[X]$.  

Furthermore, if $f:X\to Y$ is any function, $x \in X$ and $y \in Y$, we will denote by $f[x/y]$ the unique function $g:X\to Y$ such that, for a given $z \in X$ we have:
$$
g(z):=\begin{cases}
	y,\text{ if }z = x;\\
	f(z),\text{ otherwise. }
\end{cases}
$$ 
Despite our efforts to accommodate and represent the intuitionistic reading of the classical semantics for conditional logic, the meta-logic of this paper remains classical. Therefore, we presuppose the basic acquaintance with the language of first-order logic and its classical model theory. Finally, in view of the fact that this paper is about \textit{intuitionistic} conditional logic, we presuppose the basic acquaintance of the reader with both propositional and first-order version of this logic; in particular the reader should know of at least one complete Hilbert-style axiomatization of intuitionistic logic.

\section{$\mathsf{IntCK}$, the basic intuitionistic logic of conditionals}\label{S:system}
\subsection{Language and semantics}\label{sub:lands}
We are going to consider the language $\mathcal{L}$ based on a countably infinite set of propositional variables $Var$ and the following set of logical symbols $\{\bot, \top, \vee, \wedge, \to, \boxto,  \diamondto\}$. 
We will denote the propositional variables by letters $p, q, r, s$ and the formulas in $\mathcal{L}$ by $\phi, \psi, \chi, \theta$, adding the subscripts and superscripts when it is convenient.

We will use $\neg\phi$ as an abbreviation for $\phi \to \bot$ and $\phi \leftrightarrow \psi$ as an abbreviation for $(\phi \to \psi)\wedge(\psi\to \phi)$. The formulas of $\mathcal{L}$ are interpreted by the following models:

\begin{definition}\label{D:chellas-model}
	A model is a structure of the form $\mathcal{M} = (W, \leq, R, V)$, where $W \neq \emptyset$ is a set of worlds, $\leq$ is a pre-order (i.e., a reflexive and transitive relation) on $W$, $V:Var\to \mathcal{P}(W)$ is such that, for every $p \in Var$ and all $w, v \in W$ it is true that $
	(w\mathrel{\leq}v \& w \in V(p))\Rightarrow  v \in V(p)$.
	
	Next, we must have $R \subseteq W \times \mathcal{P}(W)\times W$. Thus, for every $X \subseteq W$, $R$ induces a binary relation $R_X$ on $W$ such that, for all $w,v \in W$, $R_X(w,v)$ iff $R(w, X, v)$. Finally, the following conditions must be satisfied for every $X \subseteq W$:
	\begin{align}
		(\leq^{-1}\circ R_X) &\subseteq (R_X\circ\leq^{-1})\label{Cond:1}\tag{c1}\\
		(R_X\circ\leq) &\subseteq (\leq\circ R_X)\label{Cond:2}\tag{c2}
	\end{align}	
\end{definition}
Conditions \eqref{Cond:1} and \eqref{Cond:2} can be naturally reformulated as requirements to complete the dotted parts of each of the following commuting diagrams once the respective straight-line part is given:
\begin{diagram}
w'& \rDotsto_{R_X} & v'         & & w'& \rDotsto_{R_X} & v'\\
\uTo_\leq &     & \uDotsto_\leq & & \uDotsto_\leq & &  \uTo_\leq\\
w & \rTo_{R_X} & v              & & w        &  \rTo_{R_X} & v
\end{diagram}
The models defined are the so-called \textit{Chellas models}. The alternatives to Chellas models include \textit{Segerberg models} which use a designated family of subsets of $W$ in place of its full powerset $\mathcal{P}(W)$ as in Definition \ref{D:chellas-model}. Yet another alternative is to use families of formula-indexed binary relations $\{R_\phi\mid\phi\in\mathcal{L}\}$. These choices are of no import on the level of the basic conditional logic, in other words, they all induce one and the same logic both in the classical and in the intuitionistic case. We have chosen Chellas models since their definition looks short and simple; yet the Segerberg models can be perhaps ascribed a deeper foundational meaning. For example, the first-order intuitionistic theory $Th$ defined in Section \ref{sub:foil} gives the intuitionistic encoding of the Segerberg variety of classical conditional semantics rather than the Chellas one.

Our standard notation for models is $\mathcal{M} = (W, \leq, R, V)$. Any model decorations are assumed to be inherited by their components, so that, for example $\mathcal{M}_n$ always stands for  $(W_n, \leq_n, R_n, V_n)$. A \textit{pointed model} is a structure of the form $(\mathcal{M}, w)$, where $\mathcal{M}$ is a model and $w \in W$.

The formulas of $\mathcal{L}$ are interpreted over pointed models by means of the satisfaction relation $\models$ defined by the following induction on the construction of $\phi \in \mathcal{L}$:
\begin{align*}
	\mathcal{M}, w&\models \top\\
	\mathcal{M}, w&\not\models \bot\\
	\mathcal{M}, w&\models p \Leftrightarrow w \in V(p)&&\text{for $p \in Var$}\\
	\mathcal{M}, w&\models \psi \wedge \chi \Leftrightarrow \mathcal{M}, w\models \psi\text{ and }\mathcal{M}, w\models \chi\\
	\mathcal{M}, w&\models \psi \vee \chi \Leftrightarrow \mathcal{M}, w\models \psi\text{ or }\mathcal{M}, w\models \chi\\
	\mathcal{M}, w&\models \psi \to \chi \Leftrightarrow (\forall v \geq w)(\mathcal{M}, v\models \psi\Rightarrow\mathcal{M}, v\models \chi)\\
	\mathcal{M}, w&\models \psi \boxto \chi \Leftrightarrow (\forall v \geq w)(\forall u \in W)(R_{\|\psi\|_\mathcal{M}}(v, u) \Rightarrow\mathcal{M}, u\models \chi)\\
	\mathcal{M}, w&\models \psi \diamondto \chi \Leftrightarrow (\exists u \in W)(R_{\|\psi\|_\mathcal{M}}(w, u)\text{ and }\mathcal{M}, u\models \chi)	
\end{align*}
where we assume, for any given $\phi \in \mathcal{L}$, that $\|\phi\|_\mathcal{M}$ stands for the set $\{w \in W\mid \mathcal{M}, w\models\phi\}$.

Given a pair $(\Gamma, \Delta) \in \mathcal{P}(\mathcal{L})\times\mathcal{P}(\mathcal{L})$, we say that a pointed model $(\mathcal{M}, w)$ \textit{satisfies}  $(\Gamma, \Delta)$ and write $\mathcal{M}, w \models(\Gamma, \Delta)$ iff we have $(\forall\phi \in \Gamma)(\mathcal{M}, w\models \phi)\,\&\,(\forall\psi \in \Delta)(\mathcal{M}, w\not\models \phi)$.
We say that  $(\Gamma, \Delta)$ is \textit{satisfiable} iff some pointed model satisfies it, and that $\Delta$ \textit{follows} from $\Gamma$ (and write $\Gamma \models \Delta$) iff $(\Gamma, \Delta)$ is unsatisfiable. We say that $\Gamma$ is satisfiable iff $(\Gamma, \emptyset)$ is; and if $(\mathcal{M}, w)$ satisfies $(\Gamma, \emptyset)$, then we simply write $\mathcal{M}, w \models\Gamma$. If some of $\Gamma$, $\Delta$ are singletons, we may omit the figure brackets; in case some of them are empty, we may omit them altogether. We say that $\phi \in \mathcal{L}$ is \textit{valid} iff $\emptyset \models \phi$, or, in other words, iff $\models \phi$.

Given a model $\mathcal{M}$ and an $X \subseteq W$, we say that $X$ is \textit{upward-closed in $\mathcal{M}$} iff $(\forall w \in X)(\forall v \geq w)(v \in X)$. Definition \ref{D:chellas-model} clearly implies that, for every Chellas model $\mathcal{M}$ and for every $p \in Var$, the set $V(p) = \|p\|_\mathcal{M}$ is upward-closed in $\mathcal{M}$. The latter observation can be lifted to the level of arbitrary formulas:
\begin{lemma}\label{L:chellas-monotonicity}
Let $\mathcal{M}$ be a Chellas model, and let $\phi \in \mathcal{L}$. Then $\|\phi\|_\mathcal{M}$ is upward-closed in $\mathcal{M}$. 
\end{lemma}
We omit the easy proof by induction on the construction of $\phi \in \mathcal{L}$. We close this subsection with our first result about $\mathsf{IntCK}$:
\begin{proposition}\label{P:dp}
$\mathsf{IntCK}$ has Disjunction Property. In other words, for all $\phi_1,\phi_2 \in \mathcal{L}$, $\phi_1\vee\phi_2 \in \mathsf{IntCK}$ iff $\phi_1 \in \mathsf{IntCK}$ or $\phi_2 \in \mathsf{IntCK}$.	
\end{proposition}
\begin{proof}
	The right-to-left direction is trivial. As for the other direction, assume, towards contradiction, that $\phi_1\vee\phi_2 \in \mathsf{IntCK}$, but both  $\phi_1 \notin \mathsf{IntCK}$ and $\phi_2 \notin \mathsf{IntCK}$. Then we can choose pointed models $(\mathcal{M}_1,w_1)$ and $(\mathcal{M}_2,w_2)$ such that $\mathcal{M}_i,w_i\not\models_c\phi_i$ for all $i \in \{1,2\}$; we may assume, wlog, that $W_1\cap W_2 = \emptyset$. 
	We then choose an element $w$ outside $W_1\cup W_2$ and define the following pointed model $(\mathcal{M},w)$ for which we set:
	\begin{align*}
		W &:= \{w\}\cup W_1\cup W_2\\
		\leq &:= \{(w,v)\mid v \in W\} \cup \leq_1 \cup \leq_2\\
		R &:= \{(v,X,u)\mid v,u \in W_1,\,(v,X\cap W_1, u)\in R_1\}\cup \{(v,X,u)\mid v,u \in W_2,\,(v,X\cap W_2, u)\in R_2\}\\
		V(p) &:= V_1(p)\cup V_2(p)\qquad\qquad\qquad \text{for }p \in Var
	\end{align*}
We show that $\mathcal{M}$ is indeed a model. The only non-trivial part is the satisfaction of conditions \eqref{Cond:1} and \eqref{Cond:2} from Definition \ref{D:chellas-model}.

As for \eqref{Cond:1}, assume that some $v',v,u \in W$ and $X\subseteq W$ are such that $v'\geq v\mathrel{R_X}u$. Then, by defintion of $R$, we must have either $v,u \in W_1$ or $v,u\in W_2$. Assume, wlog, that $v,u \in W_1$. Then we must have, first, that $v\mathrel{(R_1)_{X\cap W_1}}u$, and, second, that $v'\geq_1 v$ so that also $v'\in W_1$. But then, since $\mathcal{M}_1$ satisfies  \eqref{Cond:1}, there must be a $u'\in W_1$ such that $v'\mathrel{(R_1)_{X\cap W_1}}u'\geq_1u$ whence clearly also  $v'\mathrel{R_{X}}u'\geq u$, so that  \eqref{Cond:1} is shown to hold for $\mathcal{M}$. We argue similarly for  \eqref{Cond:2}.

Next, the following claim can be shown by a straightforward induction on the construction of $\phi \in\mathcal{L}$:

\textit{Claim}. For every $i\in \{1,2\}$, every $v \in W_i$, and every $\phi \in \mathcal{L}$, we have $\mathcal{M}, v\models_c\phi$ iff $\mathcal{M}_i, v\models_c\phi$. 

It follows now that $\mathcal{M}, w_i\not\models\phi_i$ for all $i \in \{1,2\}$, and, since we have $w\leq w_1, w_2$, Lemma \ref{L:chellas-monotonicity} implies that  $\mathcal{M}, w\not\models\phi_i$ for all $i \in \{1,2\}$, or, equivalently, that $\mathcal{M}, w\not\models_c\phi_1\vee\phi_2$, contrary to our assumption. The obtained contradiction shows that $\mathsf{IntCK}$ must have Disjunction Property.
\end{proof}

\subsection{Axiomatization}\label{sub:axiomatization}
In this subsection, we obtain a sound and (strongly) complete axiomatization of $\mathsf{IntCK}$. We consider the Hilbert-style axiomatic system $\mathbb{ICK}$, given by the following list of axiomatic schemes:
\begin{align}
	\text{A complete list of axioms }&\text{of intuitionistic propositional logic $\mathsf{Int}$}\label{E:a0}\tag{A0}\\
	((\phi \boxto \psi)\wedge(\phi \boxto \chi))&\leftrightarrow(\phi \boxto (\psi \wedge \chi))\label{E:a1}\tag{A1}\\
	((\phi\diamondto\psi)\wedge (\phi \boxto \chi))&\to	(\phi \diamondto(\psi \wedge\chi))\label{E:a2}\tag{A2}\\
	(\phi\diamondto(\psi \vee \chi)) &\leftrightarrow ((\phi\diamondto\psi)\vee(\phi\diamondto\chi))\label{E:a3}\tag{A3}\\
	((\phi \diamondto \psi)\to(\phi \boxto \chi))&\to(\phi \boxto (\psi \to \chi))\label{E:a4}\tag{A4}\\
	\phi&\boxto \top \label{E:a5}\tag{A5}\\
	\neg(\phi&\diamondto \bot) \label{E:a6}\tag{A6}
\end{align}
Besides the axioms, $\mathbb{ICK}$ includes the following rules of inference:
\begin{align}
	\text{From }\phi, \phi \to \psi&\text{ infer }\psi\label{E:mp}\tag{MP}\\
	\text{From }\phi\leftrightarrow\psi &\text{ infer } (\phi\boxto\chi)\leftrightarrow(\psi\boxto\chi)\label{E:RAbox}\tag{RA$\Box$}\\
	\text{From }\phi\leftrightarrow\psi &\text{ infer } (\chi\boxto\phi)\leftrightarrow(\chi\boxto\psi)\label{E:RCbox}\tag{RC$\Box$}\\
	\text{From }\phi\leftrightarrow\psi &\text{ infer } (\phi\diamondto\chi)\leftrightarrow(\psi\diamondto\chi)\label{E:RAdiam}\tag{RA$\Diamond$}
	\\
	\text{From }\phi\leftrightarrow\psi &\text{ infer } (\chi\diamondto\phi)\leftrightarrow(\chi\diamondto\psi)\label{E:RCdiam}\tag{RC$\Diamond$}		
\end{align}
We assume the standard notion of a proof in a Hilbert-style axiomatic system for $\mathbb{ICK}$, namely as a finite sequence of formulas, where every formula is either an axiom or is derived from earlier formulas by an application of a rule. A proof is a proof of its last formula. A derivation from premises is a finite sequence of formulas, where every formula is either a premise, or a provable formula, or is derived from earlier formulas by an application of \eqref{E:mp}. Given a $\Gamma\cup\{\psi\}\subseteq \mathcal{L}$, we write $\Gamma \vdash \psi$ iff $\psi$ there exists a derivation of $\psi$ from some (possibly zero) elements of $\Gamma$ as premises; we omit the figure brackets in case $\Gamma$ is either a singleton or empty. On the other hand, we write $\Gamma \vDdash \psi$ iff there is a derivable rule allowing to infer $\phi$ from $\Gamma$, i.e. iff there is a finite sequence of formulas in which the last formula is $\psi$, and every formula in the sequence is either in  $\Gamma$, or a provable formula, or is obtained from earlier formulas by an application of one of the inference rules. Thus, we have $\phi\leftrightarrow\psi \vDdash (\phi\boxto\chi)\leftrightarrow(\psi\boxto\chi)$ but $\phi\leftrightarrow\psi \not\vdash (\phi\boxto\chi)\leftrightarrow(\psi\boxto\chi)$. In particular, it follows from these conventions that $\phi$ is provable iff $\vDdash \phi$ iff $\vdash \phi$.

Before we go on to prove the soundness and completeness of $\mathbb{ICK}$ relative to $\mathsf{IntCK}$, we would like to quickly address the relations between $\mathbb{ICK}$ and the intuitionistic propositional logic $\mathsf{Int}$. 

The language $\mathcal{L}_i$ of $\mathsf{Int}$ is the $\{\boxto, \diamondto\}$-free fragment of $\mathcal{L}$; a complete axiomatization of $\mathsf{Int}$ is provided by \eqref{E:a0} together with \eqref{E:mp}. $\mathsf{Int}$ is one of the best researched non-classical propositional logics with numerous detailed expositions to be found in the existing literature (see, e.g., \cite[Ch. 5]{vandalen} for a quick textbook level introduction). The following lemma sums up the relations between $\mathsf{Int}$ and $\mathbb{ICK}$: 
\begin{lemma}\label{L:int}
	The following statements hold:
	\begin{enumerate}
		\item If $\Gamma, \Delta \subseteq\mathcal{L}_i$ are such that $\Gamma\models_{\mathsf{Int}}\Delta$, and  $\Gamma', \Delta' \subseteq \mathcal{L}$ are obtained from $\Gamma, \Delta$ by a simultaneous substitution of $\mathcal{L}$-formulas for variables, then $\Gamma'\vdash\Delta'$. Moreover, Deduction Theorem holds for $\mathbb{ICK}$ in that for all $\Gamma \cup \{\phi,\psi\}\subseteq \mathcal{L}$ we have $\Gamma \vdash \phi\to\psi$ iff $\Gamma, \phi\vdash \psi$.
		
		\item If $\phi \in \mathcal{L}_i$, then $\vdash \phi$ iff $\phi\in \mathsf{Int}$
\end{enumerate}
\end{lemma}
\begin{proof}[Proof (a sketch)]
	Part 1 is trivial. As for Part 2, its ($\Leftarrow$)-part is also trivial, and its ($\Rightarrow$)-part follows from an observation that, given a proof of $\phi \in \mathcal{L}_i$ in $\mathbb{ICK}$ we can turn it into a proof in $\mathsf{Int}$ by replacing all its subformulas of the form $\chi\boxto\theta$ with  $\top$ and all its subformulas of the form $\chi\diamondto\theta$ with $\bot$.  
\end{proof}
Turning now to the question of soundness and completeness of $\mathbb{ICK}$ relative to $\mathsf{IntCK}$, we observe, first, that  $\mathbb{ICK}$ only allows us to deduce theorems of $\mathsf{IntCK}$:
\begin{lemma}\label{L:soundness}
	For every $\phi\in\mathcal{L}$, if $\vdash\phi$, then $\phi\in\mathsf{IntCK}$.
\end{lemma}
The proof proceeds by the usual method, i.e. we show that all the axioms are valid and that the rules of $\mathbb{ICK}$ preserve the validity. We are now going to show the converse of Lemma \ref{L:soundness}, and we start our work by proving some theorems and derived rules in $\mathbb{ICK}$, which we collect in the following lemma: 
\begin{lemma}\label{L:theorems}
	Let $\phi, \psi, \chi \in \mathcal{L}$. The following theorems and derived rules can be deduced in $\mathbb{ICK}$:
	\begin{align}
		\phi &\vDdash (\psi\boxto \phi)\label{E:Rnec}\tag{Nec}\\
		(\phi \to \psi) &\vDdash ((\chi\boxto\phi)\to(\chi\boxto\psi))\label{E:Rmbox}\tag{RM$\Box$}\\
		(\phi \to \psi) &\vDdash ((\chi\diamondto\phi)\to(\chi\diamondto\psi))\label{E:Rmdiam}\tag{RM$\Diamond$}\\
		(\phi\boxto(\psi \to \chi))&\to((\phi\boxto\chi)\to(\phi\boxto\chi))\label{E:T1}\tag{T1}\\
		(\phi\boxto(\psi \to \chi))&\to((\phi\diamondto\psi)\to(\phi\diamondto\chi))\label{E:T2}\tag{T2}\\
		(\phi\boxto\psi)&\to((\phi\diamondto(\psi\to\chi))\to(\phi\diamondto\chi))\label{E:T3}\tag{T3}\\
		\neg(\phi\diamondto\psi)&\leftrightarrow(\phi\boxto\neg\psi)\label{E:T4}\tag{T4}
	\end{align}
\end{lemma} 
The sketch of its proof is relegated to Appendix \ref{A:1}.

A bi-set $(\Gamma, \Delta)\in \mathcal{P}(\mathcal{L})\times\mathcal{P}(\mathcal{L})$ is called \textit{consistent} iff for no $\Delta' \Subset \Delta$ do we have that $\Gamma \vdash \bigvee\Delta'$.\footnote{As for the limiting cases, we assume that $\bigvee\emptyset = \bot$ and that $\bigwedge\emptyset = \top$.} Note that, since $\mathbb{ICK}$ extends $\mathsf{Int}$, and also in view of Lemma \ref{L:int}.1, this definition allows for the following equivalent form:
\begin{lemma}\label{L:alt-consistency}
	A bi-set $(\Gamma, \Delta)\in \mathcal{P}(\mathcal{L})\times\mathcal{P}(\mathcal{L})$ is inconsistent iff, for some $m,n\in \omega$ some $\phi_1,\ldots,\phi_n\in \Gamma$ and some $\psi_1,\ldots,\psi_m\in \Delta$ we have: $
	\bigwedge^n_{i = 1}\phi_i\vdash\bigvee^m_{j = 1}\psi_j$, or, equivalently, $\vdash
	\bigwedge^n_{i = 1}\phi_i\to\bigvee^m_{j = 1}\psi_j$.
\end{lemma}
Furthermore, the bi-set $(\Gamma, \Delta)$ is called \textit{complete} iff $\Gamma\cup\Delta = \mathcal{L}$; and it is called \textit{maximal} iff it is both complete and consistent. The next two lemmas present some properties of the consistent and maximal bi-sets, respectively:
\begin{lemma}\label{L:consistent}
	Let	$(\Gamma, \Delta)\in \mathcal{P}(\mathcal{L})\times\mathcal{P}(\mathcal{L})$ be consistent. Then the following statements hold:
	\begin{enumerate}
		\item For every $\phi \in \mathcal{L}$, either $(\Gamma \cup \{\phi\}, \Delta)$ or $(\Gamma \cup \{\phi\}, \Delta)$ is consistent.
		
		\item For every $\phi \to \psi \in \Delta$, $(\Gamma \cup \{\phi\}, \{\psi\})$ is consistent.
		
		\item For every $\phi\boxto\psi \in \Delta$, the bi-set $(\{\chi\mid\phi\boxto\chi \in \Gamma\},\{\psi\})$ is consistent.
		
		\item For every $\phi\diamondto\psi \in \Gamma$, the bi-set $(\{\psi\}\cup\{\chi\mid\phi\boxto\chi \in \Gamma\},\{\theta\mid\phi\diamondto\theta\in\Delta\})$ is consistent.
	\end{enumerate}
\end{lemma}
\begin{proof}
Parts 1 and 2 are proved as in the case of intuitionistic propositional logic. As for Part 3, assume that  $\phi\boxto\psi \in \Delta$, and assume, towards  contradiction, that the bi-set $(\{\chi\mid\phi\boxto\chi \in \Gamma\},\{\psi\})$ is inconsistent. Then there must be $\phi\boxto\chi_1,\ldots,\phi\boxto\chi_n \in \Gamma$ such that, for $\chi:= \bigwedge^n_{i= 1}\chi_i$, we have $\chi\vdash \psi$. We reason as follows:
\begin{align}
	&\Gamma\vdash\phi\boxto\chi\label{E:con3}&&\text{by \eqref{E:a1}}\\
	&\vdash\chi\to\psi\label{E:con4}&&\text{by Lemma \ref{L:int}.1}\\
	&\vdash(\phi\boxto\chi)\to(\phi\boxto\psi)\label{E:con5}&&\text{by \eqref{E:con4}, \eqref{E:Rmbox}}\\
	&\Gamma\vdash\phi\boxto\psi\label{E:con6}&&\text{by \eqref{E:con3},\eqref{E:con5}}
\end{align}
The assumption that $\phi\boxto\psi \in \Delta$, together with \eqref{E:con6}, clearly contradicts the consistency of $(\Gamma,\Delta)$. The obtained contradiction shows that $(\{\chi\mid\phi\boxto\chi \in \Gamma\},\{\psi\})$ must be consistent.

Finally, as for Part 4, assume that $\phi\diamondto\psi \in \Gamma$, and assume, towards contradiction, that the bi-set $(\{\psi\}\cup\{\chi\mid\phi\boxto\chi \in \Gamma\},\{\theta\mid\phi\diamondto\theta\in\Delta\})$ is inconsistent. Then there must be some $\phi\boxto\chi_1,\ldots,\phi\boxto\chi_n \in \Gamma$ and $\phi\diamondto\theta_1,\ldots,\phi\diamondto\theta_m \in \Delta$ such that, for $\chi:= \bigwedge^n_{i= 1}\chi_i$ and $\theta:= \bigvee^m_{j= 1}\theta_j$,  we have $\chi,\psi\vdash \theta$. We then reason as follows:
\begin{align}
	&\Gamma\vdash\phi\boxto\chi\label{E:cons4}&&\text{by \eqref{E:a1}}\\
	&\vdash\chi\to(\psi\to \theta)\label{E:cons5}&&\text{by Lemma \ref{L:int}.1}\\
	&\vdash(\phi\boxto\chi)\to(\phi\boxto(\psi\to \theta))\label{E:cons6}&&\text{by \eqref{E:cons5}, \eqref{E:Rmbox}}\\
	&\Gamma\vdash(\phi\diamondto\psi)\to (\phi\diamondto\theta)\label{E:cons8}&&\text{by \eqref{E:cons4},\eqref{E:cons6},\eqref{E:T2}}\\
	&\Gamma\vdash\phi\diamondto\theta\label{E:cons9}&&\text{by \eqref{E:cons8},$\phi\diamondto\psi \in \Gamma$}\\
	&\Gamma\vdash\phi\diamondto\theta_1\vee\ldots\vee\phi\diamondto\theta_m\label{E:cons10}&&\text{by \eqref{E:cons9}, \eqref{E:a3}}
\end{align}
It follows now from \eqref{E:cons10} that $(\Gamma,\Delta)$ must be inconsistent, which contradicts our initial assumption. The obtained contradiction shows that the bi-set $(\{\psi\}\cup\{\chi\mid\phi\boxto\chi \in \Gamma\},\{\theta\mid\phi\diamondto\theta\in\Delta\})$ must be, in fact, consistent. 
\end{proof} 
\begin{lemma}\label{L:maximal}
	Let	$(\Gamma, \Delta), (\Gamma_0,\Delta_0), (\Gamma_1,\Delta_1) \in \mathcal{P}(\mathcal{L})\times\mathcal{P}(\mathcal{L})$ be maximal, let $\phi,\psi\in\mathcal{L}$. Then the following statements are true:
	\begin{enumerate}
		\item If $\Gamma\vdash\phi$, then $\phi\in \Gamma$.
		
		\item $\phi\wedge\psi\in\Gamma$ iff $\phi, \psi\in \Gamma$.
		
		\item $\phi\vee\psi \in \Gamma$ iff $\phi \in \Gamma$ or $\psi\in\Gamma$.
		
		\item If $\phi\to\psi, \phi \in \Gamma$, then $\psi \in \Gamma$.
		
		\item If $\Gamma_0 \subseteq \Gamma$, $\{\psi\mid \phi\boxto\psi\in \Gamma_0\}\subseteq \Gamma_1$, and $\{\phi\diamondto\psi\mid \psi\in \Gamma_1\}\subseteq \Gamma_0$ then
		 
		\noindent$(\Gamma_1 \cup \{\psi\mid\phi\boxto\psi\in \Gamma\}, \{\psi\mid\phi\diamondto\psi\in \Delta\})$ is consistent.
		
		\item If $\Gamma_1 \subseteq \Gamma$, $\{\psi\mid \phi\boxto\psi\in \Gamma_0\}\subseteq \Gamma_1$, and $\{\phi\diamondto\psi\mid \psi\in \Gamma_1\}\subseteq \Gamma_0$ then
		 
		\noindent$(\Gamma_0 \cup \{\phi\diamondto\psi\mid\psi\in \Gamma\}, \{\phi\boxto\psi\mid\psi\in \Delta\})$ is consistent.
	\end{enumerate}
\end{lemma}
\begin{proof}
	The Parts 1--4 are handled as in the case of $\mathsf{Int}$. E.g., for Part 3 observe that, if $\phi\vee\psi \in \Gamma$ and $\phi,\psi\in \Delta$, then we must have $\Gamma\vdash\phi\vee\psi$, thus contradicting the consistency of $(\Gamma, \Delta)$.
	
	As for Part 5, assume its hypothesis and suppose, towards contradiction that $(\Gamma_1 \cup \{\psi\mid\phi\boxto\psi\in \Gamma\}, \{\psi\mid\phi\diamondto\psi\in \Delta\})$ is inconsistent. Then there must exist some $\psi_1,\ldots,\psi_n \in \Gamma_1$, $\phi\boxto\chi_1,\ldots,\phi\boxto\chi_m\in \Gamma$ and some $\phi\diamondto\theta_1,\ldots,\phi\diamondto\theta_k\in \Delta$, such that, for $\psi:= \bigwedge^n_{i=1}\psi_i$, $\chi:= \bigwedge^m_{j=1}\chi_j$, and $\theta:= \bigvee^k_{r=1}\theta_r$ we have $\psi,\chi\vdash\theta$. But then:
	\begin{align}
		&\psi\vdash\chi\to\theta\label{E:max4}&&\text{Lemma \ref{L:int}.1}\\
		&\chi\to\theta\in \Gamma_1\label{E:max6}&&\text{\eqref{E:max4}, Part 1}\\
		&\phi\diamondto(\chi\to\theta)\in \Gamma_0\label{E:max7}&&\text{\eqref{E:max6}, $\{\phi\diamondto\psi\mid \psi\in \Gamma_1\}\subseteq \Gamma_0$}\\
		&\phi\diamondto(\chi\to\theta)\in \Gamma\label{E:max8}&&\text{\eqref{E:max7}, $\Gamma_0\subseteq \Gamma_1$}\\
	&\Gamma\vdash (\phi\boxto\chi)\to(\phi\diamondto\theta)\label{E:max10}&&\text{\eqref{E:max8}, \eqref{E:T3}}\\
	&\Gamma\vdash\phi\boxto\chi\label{E:max11}&&\text{\eqref{E:a1}}\\
	&\Gamma\vdash\phi\diamondto\theta\label{E:max12}&&\text{\eqref{E:max10}, \eqref{E:max11}}\\
	&\Gamma\vdash(\phi\diamondto\theta_1)\vee\ldots\vee(\phi\diamondto\theta_k)\label{E:max13}&&\text{\eqref{E:max12}, \eqref{E:a3}}
\end{align}
It follows now from \eqref{E:max13}, that $(\Gamma, \Delta)$ is not consistent and thus also not maximal, contrary to our initial assumption. The obtained contradiction shows that the bi-set 

\noindent$(\Gamma_1 \cup \{\psi\mid\phi\boxto\psi\in \Gamma\}, \{\psi\mid\phi\diamondto\psi\in \Delta\})$ must have been consistent.

For Part 6,  assume its hypothesis and suppose that $(\Gamma_0 \cup \{\phi\diamondto\psi\mid\psi\in \Gamma\}, \{\phi\boxto\psi\mid\psi\in \Delta\})$ is inconsistent. Then there must exist some $\psi_1,\ldots,\psi_n \in \Gamma_0$, $\chi_1,\ldots,\chi_m\in \Gamma$, and some $\theta_1,\ldots,\theta_k\in \Delta$, such that
$\bigwedge^n_{i=1}\psi_i,\bigwedge^m_{j = 1}(\phi\diamondto\chi_j)\vdash\bigvee^k_{r = 1}(\phi\boxto\theta_r)$.
Again, we set $\psi:= \bigwedge^n_{i=1}\psi_i$, $\chi:= \bigwedge^m_{j=1}\chi_j$, and $\theta:= \bigvee^k_{r=1}\theta_r$, and reason as follows:
\begin{align}
	&\Gamma_0\vdash\bigwedge^m_{j = 1}(\phi\diamondto\chi_j)\to\bigvee^k_{r = 1}(\phi\boxto\theta_r)\label{E:maxi4}&&\text{Lemma \ref{L:int}.1}\\
	&\vdash(\phi\diamondto\chi)\to\bigwedge^m_{j = 1}(\phi\diamondto\chi_j)\label{E:maxi7}&&\text{\eqref{E:a0}, \eqref{E:mp}, \eqref{E:Rmdiam}}\\
	&\vdash\bigvee^k_{r = 1}(\phi\boxto\theta_j)\to(\phi\boxto \theta)\label{E:maxi9}&&\text{\eqref{E:a0}, \eqref{E:mp}, \eqref{E:Rmbox}}\\
	&\Gamma_0\vdash(\phi\diamondto\chi)\to(\phi\boxto \theta)\label{E:maxi10}&&\text{\eqref{E:maxi4},\eqref{E:maxi7},\eqref{E:maxi9}}\\
	&\phi\boxto(\chi\to\theta)\in\Gamma_0\label{E:maxi12}&&\text{\eqref{E:maxi10},\eqref{E:a4}, Part 1}\\
	&(\chi\to\theta)\in\Gamma_1\label{E:maxi13}&&\text{\eqref{E:maxi12}, $\{\psi\mid \phi\boxto\psi\in \Gamma_0\}\subseteq \Gamma_1$}
\end{align}
By $\Gamma_1\subseteq \Gamma$, we know then that also $(\chi\to\theta)\in\Gamma$. Now, since clearly $\Gamma\vdash\chi$, it follows that $\Gamma\vdash\theta$, whence, by Parts 1 and 3, we know that $\Gamma\vdash\theta_r$ for some $1\leq r \leq k$, which clearly contradicts the consistency of $(\Gamma, \Delta)$. The obtained contradiction shows that the bi-set $(\Gamma_0 \cup \{\phi\diamondto\psi\mid\psi\in \Gamma\}, \{\phi\boxto\psi\mid\psi\in \Delta\})$ must have been consistent.
\end{proof}
We observe, next, that we can use the usual Lindenbaum construction to extend every consistent bi-set to a maximal one:
\begin{lemma}\label{L:lindenbaum}
	Let $(\Gamma, \Delta)\in \mathcal{P}(\mathcal{L})\times\mathcal{P}(\mathcal{L})$ be consistent. Then there exists a maximal $(\Xi, \Theta)\in \mathcal{P}(\mathcal{L})\times\mathcal{P}(\mathcal{L})$ such that $\Gamma \subseteq \Xi$ and $\Delta \subseteq \Theta$. 
\end{lemma}
Next, we define the canonical model $\mathcal{M}_c$ for $\mathbb{ICK}$:
\begin{definition}\label{D:canonical-model}
	The structure $\mathcal{M}_c$ is the tuple $(W_c, \leq_c, R_c, V_c)$ such that:
	\begin{itemize}
		\item $W_c:=\{(\Gamma, \Delta)\in \mathcal{P}(\mathcal{L})\times\mathcal{P}(\mathcal{L})\mid (\Gamma, \Delta)\text{ is maximal}\}$.
		
		\item $(\Gamma_0,\Delta_0)\leq_c(\Gamma_1,\Delta_1)$ iff $\Gamma_0\subseteq\Gamma_1$ for all $(\Gamma_0,\Delta_0),(\Gamma_1,\Delta_1)\in W_c$.
		
		\item For all $(\Gamma_0,\Delta_0),(\Gamma_1,\Delta_1)\in W_c$ and $X \subseteq W_c$, we have $((\Gamma_0,\Delta_0),X,(\Gamma_1,\Delta_1)) \in R_c$ iff there exists a $\phi\in\mathcal{L}$, such that all of the following holds:
		\begin{itemize}
			\item $X = \{(\Gamma,\Delta)\in W_c\mid\phi\in\Gamma\}$.
			
			\item $\{\psi\mid\phi\boxto\psi\in \Gamma_0\}\subseteq \Gamma_1$.
			
			\item $\{\phi\diamondto\psi\mid\psi\in \Gamma_1\}\subseteq \Gamma_0$.
		\end{itemize}
		
		\item $V_c(p):=\{(\Gamma,\Delta)\in W_c\mid p\in\Gamma\}$ for every $p \in Var$.	
	\end{itemize}
\end{definition}
First of all, we observe that the definition of $R_c$ does not depend on the choice of the representative formula $\phi\in\mathcal{L}$. The following lemma provides the necessary stepping stone:
\begin{lemma}\label{L:representatives}
	Let $\phi,\psi \in \mathcal{L}$ be such that $\{(\Gamma,\Delta)\in W_c\mid\phi\in\Gamma\} = \{(\Gamma,\Delta)\in W_c\mid\psi\in\Gamma\}$. Then, for every $(\Gamma',\Delta')\in W_c$ and every $\chi\in\mathcal{L}$ we have:
	\begin{enumerate}
		\item $\phi\boxto\chi\in\Gamma'\Leftrightarrow\psi\boxto\chi\in\Gamma'$.
		
		\item $\phi\diamondto\chi\in\Gamma'\Leftrightarrow\psi\diamondto\chi\in\Gamma'$.
	\end{enumerate}
\end{lemma}
\begin{proof}
	Assume the hypothesis of the Lemma. We will show that in this case we must have $\vdash\phi\leftrightarrow\psi$. Suppose not, and assume, wlog, that $\not\vdash\phi\to\psi$. Then $(\{\phi\},\{\psi\})$ must be consistent and thus extendable to some maximal $(\Gamma_0,\Delta_0)\supseteq (\{\phi\},\{\psi\})$. But then clearly
	$$
	(\Gamma_0,\Delta_0)\in \{(\Gamma,\Delta)\in W_c\mid\phi\in\Gamma\} \setminus \{(\Gamma,\Delta)\in W_c\mid\psi\in\Gamma\},
	$$ 
	in contradiction with our initial assumptions. The obtained contradiction shows that we must have $\vdash\phi\leftrightarrow\psi$. The application of \eqref{E:RAbox} and \eqref{E:RAdiam} then yields that also $\vdash(\phi\boxto\chi)\leftrightarrow\psi\boxto\chi$
	and $\vdash(\phi\diamondto\chi)\leftrightarrow\psi\diamondto\chi$
	for every $\chi\in\mathcal{L}$, whence our Lemma clearly follows.  
\end{proof}
We have to make sure that we have indeed just defined a model:
\begin{lemma}\label{L:canonical-model}
The structure $\mathcal{M}_c$, as given in Definition \ref{D:canonical-model}, is a model.	
\end{lemma}
\begin{proof}
	We show, first, the $W_c \neq \emptyset$. Indeed, consider the bi-set $(\emptyset,\emptyset)$. It is well-known that $\bot\notin\mathsf{Int}$, and since $\bot\in\mathcal{L}_i$, Lemma \ref{L:int}.2 now implies that $\bot\notin\mathbb{ICK}$, whence, by Lemma \ref{L:int}.1, we know that $\not\vdash\bot$. But the latter means that $(\emptyset,\emptyset)$ is consistent. Therefore, by Lemma \ref{L:lindenbaum}, there must exist a maximal $(\Gamma, \Delta)\supseteq(\emptyset,\emptyset)$; and we will have, by Definition \ref{D:canonical-model}, that $(\Gamma, \Delta)\in W_c$.
	
	It is also clear from Definition  \ref{D:canonical-model} and Lemma \ref{L:representatives} that $\leq_c$ is a pre-order, and that $R_c\subseteq W_c\times\mathcal{P}(W_c)\times W_c$ is well-defined. So it only remains to check the satisfaction of conditions \eqref{Cond:1} and \eqref{Cond:2} from Definition \ref{D:chellas-model}.
	
	As for \eqref{Cond:1}, assume that $(\Gamma,\Delta)$,  $(\Gamma_0,\Delta_0)$, and  $(\Gamma_1,\Delta_1)$ are maximal, and that $X\subseteq W_c$ is such that we have $(\Gamma,\Delta) \mathrel{\geq_c}(\Gamma_0,\Delta_0)\mathrel{(R_c)_X}(\Gamma_1,\Delta_1)$. Then, in particular, $\Gamma\supseteq\Gamma_0$. Moreover, we can choose a $\phi \in\mathcal{L}$ such that all of the following holds:
	\begin{align}
		&X = \{(\Xi,\Theta)\in W_c\mid\phi\in\Xi\}\label{E:mod1}\\
		&\{\psi\mid\phi\boxto\psi\in \Gamma_0\}\subseteq \Gamma_1\label{E:mod2}\\
		&\{\phi\diamondto\psi\mid\psi\in \Gamma_1\}\subseteq \Gamma_0\label{E:mod3}
	\end{align}
By Lemma \ref{L:maximal}.5, the bi-set $(\Gamma_1 \cup \{\psi\mid\phi\boxto\psi\in \Gamma\}, \{\psi\mid\phi\diamondto\psi\in \Delta\})$ must then be consistent, so that, by Lemma \ref{L:lindenbaum}, this bi-set must then be extendable to some maximal bi-set $(\Gamma',\Delta')$. In particular, we will have $\Gamma'\supseteq \Gamma_1$ whence clearly $(\Gamma',\Delta')\mathrel{\geq_c}(\Gamma_1,\Delta_1)$.

Next, we get that $\{\psi\mid\phi\boxto\psi\in \Gamma\}\subseteq \Gamma'$ trivially by the choice of $(\Gamma',\Delta')$. Moreover, if $\psi \in \Gamma'$, then we must have $\psi\notin\Delta'$ by the consistency of $(\Gamma',\Delta')$. But this means, in particular, that we cannot have $\phi\diamondto\psi\in\Delta$, so that, by the completeness of $(\Gamma,\Delta)$ we must have $\phi\diamondto\psi\in \Gamma$. Thus we have shown that also $\{\phi\diamondto\psi\mid\psi\in \Gamma'\}\subseteq \Gamma$. Summing this up with \eqref{E:mod1}, we obtain that $(\Gamma,\Delta)\mathrel{(R_c)_X}(\Gamma',\Delta')$. Thus we get that $
(\Gamma,\Delta)\mathrel{(R_c)_X}(\Gamma',\Delta')\mathrel{\geq_c}(\Gamma_1,\Delta_1)$, and condition \eqref{Cond:1} is shown to be satisfied.

As for \eqref{Cond:2}, assume that $(\Gamma,\Delta)$,  $(\Gamma_0,\Delta_0)$, and  $(\Gamma_1,\Delta_1)$ are maximal, and that $X\subseteq W_c$ is such that we have $(\Gamma_0,\Delta_0)\mathrel{(R_c)_X} (\Gamma_1,\Delta_1)\mathrel{\leq_c}(\Gamma,\Delta)$. Then, in particular, $\Gamma\supseteq\Gamma_1$. Moreover, we can choose a $\phi \in\mathcal{L}$ such that all of \eqref{E:mod1}--\eqref{E:mod3} hold.

By Lemma \ref{L:maximal}.6, the bi-set $(\Gamma_0 \cup \{\phi\diamondto\psi\mid\psi\in \Gamma\}, \{\phi\boxto\psi\mid\psi\in \Delta\})$ must then be consistent, so that, by Lemma \ref{L:lindenbaum}, this bi-set must then be extendable to some maximal bi-set $(\Gamma',\Delta')$. In particular, we will have $\Gamma'\supseteq \Gamma_0$ whence clearly $(\Gamma',\Delta')\mathrel{_c\geq}(\Gamma_0,\Delta_0)$.

Next, assume that $\psi\in \mathcal{L}$ is such that $\phi\boxto\psi\in \Gamma'$. If $\psi\notin\Gamma$, then, by the completeness of $(\Gamma,\Delta)$, we must have $\psi \in\Delta$, whence it follows that $\phi\boxto\psi\in \Delta'$. But the latter contradicts the consistency of $(\Gamma',\Delta')$. Therefore, we must have $\psi\in\Gamma$. Since the choice of $\psi$ was arbitrary, we have shown that $\{\psi\mid\phi\boxto\psi\in \Gamma'\}\subseteq \Gamma$. Moreover, if $\psi\in\Gamma$, then clearly $\phi\diamondto\psi \in \Gamma'$ so that $\{\phi\diamondto\psi\mid\psi\in \Gamma\}\subseteq \Gamma'$ also holds. Summing this up with \eqref{E:mod1}, we obtain that $(\Gamma',\Delta')\mathrel{(R_c)_X}(\Gamma,\Delta)$. Thus we get that $(\Gamma_0,\Delta_0)\mathrel{\leq_c}(\Gamma',\Delta')\mathrel{(R_c)_X}(\Gamma,\Delta)$, and condition \eqref{Cond:2} is shown to be satisfied.
\end{proof}
The truth lemma for this model then looks as follows:
\begin{lemma}\label{L:truth}
	For every $\phi\in\mathcal{L}$ and for every $(\Gamma,\Delta)\in W_c$, we have $
\mathcal{M}_c,(\Gamma,\Delta)\models\phi \Leftrightarrow \phi \in \Gamma$.
\end{lemma}
\begin{proof}
	We proceed by induction on the construction of $\phi$.
	
\textit{Basis}. If $\phi = p \in Var$, then the lemma holds by the definition of $\mathcal{M}_c$. If $\phi \in \{\top, \bot\}$, then we reason as in the case of $\mathsf{Int}$.

\textit{Induction step}. The cases associated with $\wedge$, $\vee$, and $\to$ are solved as in the case of $\mathsf{Int}$. We treat the two remaining cases:

\textit{Case 1}. $\phi = \psi \boxto \chi$.

($\Leftarrow$)  Let $(\Gamma,\Delta)\in W_c$ be such that $\phi = \psi \boxto \chi \in \Gamma$, and let $(\Gamma_0,\Delta_0), (\Gamma_1,\Delta_1)\in W_c$ be such that $(\Gamma,\Delta)\mathrel{\leq_c}(\Gamma_0,\Delta_0)\mathrel{(R_c)_{\|\psi\|_{\mathcal{M}_c}}}(\Gamma_1,\Delta_1)$. Then we must have, on the one hand, that $\Gamma \subseteq \Gamma_0$, so that, in particular, $\psi \boxto \chi \in \Gamma_0$. On the other hand, there must exist a $\theta\in \mathcal{L}$ such that all of the following holds:
\begin{align}
	&\|\psi\|_{\mathcal{M}_c} = \{(\Xi,\Theta)\in W_c\mid\theta\in\Xi\}\label{E:choice1}\\
	&\{\xi\mid\theta\boxto\xi\in \Gamma_0\}\subseteq \Gamma_1\label{E:choice2}\\
	&\{\theta\diamondto\xi\mid\xi\in \Gamma_1\}\subseteq \Gamma_0\label{E:choice3}
\end{align}
By IH, we know that also $\|\psi\|_{\mathcal{M}_c} = \{(\Xi,\Theta)\in W_c\mid\psi\in\Xi\}$. We thus get that:
\begin{equation}\label{E:choice4}
	\{(\Xi,\Theta)\in W_c\mid\psi\in\Xi\} = \{(\Xi,\Theta)\in W_c\mid\theta\in\Xi\}
\end{equation}
Since we have shown that $\psi \boxto \chi \in \Gamma_0$, we know that, by Lemma \ref{L:representatives} and \eqref{E:choice4}, we must also have $\theta\boxto\chi \in \Gamma_0$. It follows now, by \eqref{E:choice2}, that we must have $\chi\in \Gamma_1$. Next, IH implies that  $\mathcal{M}_c,(\Gamma_1,\Delta_1)\models\chi$. Since the choice of  $(\Gamma_0,\Delta_0), (\Gamma_1,\Delta_1)\in W_c$ under the condition that $(\Gamma,\Delta)\mathrel{\leq_c}(\Gamma_0,\Delta_0)\mathrel{(R_c)_{\|\psi\|_{\mathcal{M}_c}}}(\Gamma_1,\Delta_1)$ was made arbitrarily, it follows that we must have $\mathcal{M}_c,(\Gamma,\Delta)\models \psi \boxto \chi = \phi$.

($\Rightarrow$)  Let $(\Gamma,\Delta)\in W_c$ be such that $\phi = \psi \boxto \chi \notin \Gamma$. By completeness of $(\Gamma, \Delta)$, we must then have $\psi \boxto \chi \in \Delta$. Therefore, by Lemma \ref{L:consistent}.4, the bi-set $(\{\theta\mid\psi\boxto\theta\in\Gamma\},\{\chi\})$ must be consistent. By Lemma \ref{L:lindenbaum}, we can extend the latter bi-set to a maximal $(\Gamma',\Delta')\supseteq (\{\theta\mid\psi\boxto\theta\in\Gamma\},\{\chi\})$. Now, consider the bi-set $(\Gamma_0,\Delta_0):= (\Gamma\cup\{\psi\diamondto\theta\mid\theta\in \Gamma'\}, \{\psi\boxto\xi\mid\xi\in\Delta'\})$. We claim that $(\Gamma_0,\Delta_0)$ is consistent. Indeed, otherwise we can choose some $\gamma_1,\ldots,\gamma_n\in\Gamma$, $\tau_1,\ldots,\tau_m\in\Gamma'$ and $\xi_1,\ldots,\xi_k\in\Delta'$ such that $\bigwedge^n_{i = 1}\gamma_i,\bigwedge^m_{j = 1}(\psi\diamondto\tau_j)\vdash\bigvee^k_{r = 1}(\psi\boxto\xi_r)$. But then, for $\gamma:= \bigwedge^n_{i = 1}\gamma_i$, $\tau:= \bigwedge^m_{j = 1}\tau_j$, and $\xi:= \bigvee^k_{r = 1}\xi_r$, we have:
\begin{align}
	&\gamma\vdash \bigwedge^m_{j = 1}(\psi\diamondto\tau_j)\to\bigvee^k_{r = 1}(\psi\boxto\xi_r)\label{E:can5} &&\text{Lemma \ref{L:int}.1}\\
	&\vdash(\psi\diamondto\tau)\to\bigwedge^m_{j = 1}(\psi\diamondto\tau_j)\label{E:can7} &&\text{\eqref{E:a0}, \eqref{E:mp}, \eqref{E:Rmdiam}}\\
	&\vdash\bigvee^k_{r = 1}(\psi\boxto\xi_r)\to(\psi\boxto\xi)\label{E:can9} &&\text{\eqref{E:a0}, \eqref{E:mp}, \eqref{E:Rmbox}}\\
	&\gamma\vdash(\psi\diamondto\tau)\to(\psi\boxto\xi)\label{E:can10} &&\text{\eqref{E:can5}, \eqref{E:can7}, \eqref{E:can9}}\\
	&(\psi\boxto(\tau\to\xi))\in \Gamma\label{E:can12} &&\text{\eqref{E:can10}, \eqref{E:a4}, Lemma \ref{L:maximal}.1}\\
	&(\tau\to\xi)\in \Gamma'\label{E:can13} &&\text{\eqref{E:can12}, choice of $(\Gamma',\Delta')$}
\end{align}
By \eqref{E:can13}, $(\Gamma',\Delta')$ must be inconsistent, which contradicts its choice and shows that $(\Gamma_0,\Delta_0)$ must have been consistent. Therefore, $(\Gamma_0,\Delta_0)$ is extendable to a maximal bi-set $(\Gamma_1,\Delta_1)\supseteq(\Gamma_0,\Delta_0)$. 

We now claim that we have both $(\Gamma,\Delta)\mathrel{\leq_c}(\Gamma_1,\Delta_1)$ and $((\Gamma_1,\Delta_1),\{(\Xi,\Theta)\in W_c\mid\psi\in\Xi\},(\Gamma',\Delta'))\in R_c$. Indeed, the first part follows from the fact that $\Gamma_1\supseteq\Gamma_0 \supseteq\Gamma$ which is trivial by the choice of $(\Gamma_1,\Delta_1)$ and $(\Gamma_0,\Delta_0)$. As for the second part, note that (a) for every $\theta\in\mathcal{L}$, if $\psi\boxto\theta\in \Gamma_1$ and $\theta \notin \Gamma'$, then, by the completeness of $(\Gamma',\Delta')$, we must have $\theta \in \Delta'$. But then $\psi\boxto\theta\in\Delta_0\subseteq\Delta_1$, which contradicts the consistency  of $(\Gamma_1,\Delta_1)$. The obtained contradiction shows that  $\{\psi\mid\psi\boxto\psi\in \Gamma_1\}\subseteq \Gamma'$. Next, (b) we trivially get that $\{\psi\diamondto\theta\mid\theta\in \Gamma'\}\subseteq\Gamma_0 \subseteq \Gamma_1$. Summing up (a) and (b), we get that $((\Gamma_1,\Delta_1),\{(\Xi,\Theta)\in W_c\mid\psi\in\Xi\},(\Gamma',\Delta'))\in R_c$.

It remains to notice that, by IH, we must have $\|\psi\|_{\mathcal{M}_c} = \{(\Xi,\Theta)\in W_c\mid\psi\in\Xi\}$, so that we have shown, in effect that $((\Gamma_1,\Delta_1),\|\psi\|_{\mathcal{M}_c},(\Gamma',\Delta'))\in R_c$.

Observe, next, that we must also have $\chi\in \Delta'$, so that, by the consistency of $(\Gamma',\Delta')$ we must have $\chi\notin\Gamma'$, whence, by IH, $\mathcal{M}_c,(\Gamma',\Delta')\not\models \chi$. Together with the fact that $(\Gamma,\Delta)\mathrel{\leq_c}(\Gamma_1,\Delta_1)$ and $((\Gamma_1,\Delta_1),\|\psi\|_{\mathcal{M}_c},(\Gamma',\Delta'))\in R_c$, this finaly implies that $\mathcal{M}_c,(\Gamma',\Delta')\not\models \psi\boxto\chi = \phi$.

\textit{Case 2}. $\phi = \psi \diamondto \chi$.

($\Leftarrow$)  Let $(\Gamma,\Delta)\in W_c$ be such that $\phi = \psi \diamondto \chi \in \Gamma$. Then, by Lemma \ref{L:consistent}.5, the bi-set $(\{\chi\}\cup\{\xi\mid\psi\boxto\xi \in \Gamma\},\{\theta\mid\psi\diamondto\theta\in\Delta\})$ must be consistent, and, by Lemma \ref{L:lindenbaum}, there must be a maximal bi-set $(\Gamma',\Delta')\in W_c$ such that $
(\Gamma',\Delta')\supseteq (\{\chi\}\cup\{\xi\mid\psi\boxto\xi \in \Gamma\},\{\theta\mid\psi\diamondto\theta\in\Delta\})$.

The latter means that $\chi\in\Gamma'$, so that, by IH, we must have $\mathcal{M}_c,(\Gamma',\Delta')\models \chi$. On the other hand, IH yields that  $\|\psi\|_{\mathcal{M}_c} = \{(\Xi,\Theta)\in W_c\mid\psi\in\Xi\}$. Next, by the choice of  $(\Gamma',\Delta')$, we know that $\{\xi\mid\phi\boxto\xi \in \Gamma\}\subseteq \Gamma'$. Finally, if $\theta\in\Gamma'$, then, by the consistency of $(\Gamma',\Delta')$, $\theta\notin\Delta'$, whence clearly $\psi\diamondto\theta\notin\Delta$. But then, by the completeness of  $(\Gamma,\Delta)$, we must have $\psi\diamondto\theta\in\Gamma$. We have thus shown that $\{\psi\diamondto\theta\mid\theta\in\Gamma'\}\subseteq\Gamma$. Summing up, we must have $((\Gamma,\Delta),\|\psi\|_{\mathcal{M}_c},(\Gamma',\Delta'))\in R_c$, and, since we have also shown that $\mathcal{M}_c,(\Gamma',\Delta')\models \chi$, $\mathcal{M}_c,(\Gamma,\Delta)\models \psi\diamondto\chi$ clearly follows.

($\Rightarrow$) Let $(\Gamma,\Delta)\in W_c$ be such that $\phi = \psi \diamondto \chi \notin \Gamma$. Assume now that  $(\Gamma',\Delta')\in W_c$ is such that $((\Gamma,\Delta),\|\psi\|_{\mathcal{M}_c},(\Gamma',\Delta'))\in R_c$. Let $\theta\in\mathcal{L}$ be such that all of the following holds:
\begin{align}
	&\|\psi\|_{\mathcal{M}_c} = \{(\Xi,\Theta)\in W_c\mid\theta\in\Xi\}\label{E:choice5}\\
	&\{\xi\mid\theta\boxto\xi\in \Gamma\}\subseteq \Gamma'\label{E:choice6}\\
	&\{\theta\diamondto\xi\mid\xi\in \Gamma'\}\subseteq \Gamma\label{E:choice7}
\end{align}
By IH, we know that also $\|\psi\|_{\mathcal{M}_c} = \{(\Xi,\Theta)\in W_c\mid\psi\in\Xi\}$. We thus get that:
\begin{equation}\label{E:choice8}
	\{(\Xi,\Theta)\in W_c\mid\psi\in\Xi\} = \{(\Xi,\Theta)\in W_c\mid\theta\in\Xi\}
\end{equation}
Since we have $\psi \diamondto \chi \notin \Gamma$, we know that, by Lemma \ref{L:representatives} and \eqref{E:choice8}, we must also have $\theta\diamondto\chi \notin \Gamma$. It follows now, by \eqref{E:choice7}, that we must have $\chi\notin \Gamma'$. Next, IH implies that  $\mathcal{M}_c,(\Gamma',\Delta')\not\models\chi$. Since the choice of $(\Gamma',\Delta')\in W_c$ under the condition that $((\Gamma,\Delta),\|\psi\|_{\mathcal{M}_c},(\Gamma',\Delta'))\in R_c$
 was made arbitrarily, it follows that we must have$\mathcal{M}_c,(\Gamma,\Delta)\not\models \psi \diamondto \chi = \phi$.
\end{proof}
The truth lemma allows us to deduce the (strong) soundness and completeness of $\mathbb{ICK}$ relative to $\mathsf{IntCK}$ in the usual way:
\begin{theorem}\label{T:completeness}
	For every $(\Gamma,\Delta)\in \mathcal{P}(\mathcal{L})\times\mathcal{P}(\mathcal{L})$, $(\Gamma,\Delta)$ is consistent iff $(\Gamma,\Delta)$ is satisfiable. In particular, for every $\phi\in\mathcal{L}$, $\vdash\phi\Leftrightarrow(\phi\in\mathsf{IntCK})$.
\end{theorem}
\begin{proof}
	($\Leftarrow$) We argue by contraposition. First, note that we can show the following claim by induction on the length of a derivation $\psi_1,\ldots,\psi_n = \phi$ of $\phi$ from the premises in $\Gamma$:
	
	\textit{Claim}. If $\Gamma\vdash\phi$, then $(\Gamma, \{\phi\})$ is unsatisfiable.
	
	If now $(\Gamma,\Delta)\in \mathcal{P}(\mathcal{L})\times\mathcal{P}(\mathcal{L})$ is inconsistent, then we must have $\Gamma\vdash\psi_1\vee\ldots\vee\psi_n$ for some $\psi_1,\ldots,\psi_n\in \Delta$. But then the Claim implies that $(\Gamma, \{\psi_1\vee\ldots\vee\psi_n\})$ is unsatisfiable whence clearly $(\Gamma, \{\psi_1,\ldots,\psi_n\})$ is unsatisfiable, so that $(\Gamma,\Delta)\supseteq(\Gamma, \{\psi_1,\ldots,\psi_n\})$ is unsatisfiable as well.
	
	($\Rightarrow$) $(\Gamma,\Delta)\in \mathcal{P}(\mathcal{L})\times\mathcal{P}(\mathcal{L})$ is consistent, then, by Lemma \ref{L:truth}, we have $\mathcal{M}_c,(\Gamma,\Delta)\models(\Gamma,\Delta)$.
	
	We have thus shown Theorem \ref{T:completeness}. In particular we have shown that, for every $\phi \in \mathcal{L}$, $\vdash\phi$ iff $(\emptyset,\{\phi\})$ is inconsistent iff $\phi\in\mathsf{IntCK}$.
\end{proof}
As a usual corollary, we obtain the compactness of $\mathsf{IntCK}$ for bi-sets:
\begin{corollary}\label{C:compactness}
	$(\Gamma,\Delta)\in \mathcal{P}(\mathcal{L})\times\mathcal{P}(\mathcal{L})$, $(\Gamma,\Delta)$ is satisfiable iff, for every $(\Gamma',\Delta')\Subset(\Gamma,\Delta)$, $(\Gamma',\Delta')$ is satisfiable.
\end{corollary}
\begin{proof}
	The ($\Rightarrow$)-part is straightforward, as for the converse, we argue by contraposition. If $(\Gamma,\Delta)$ is unsatisfiable, then, by Theorem \ref{T:completeness}, $(\Gamma,\Delta)$ is inconsistent, therefore, for some $\psi_1,\ldots,\psi_n\in \Delta$ we have $\Gamma\vdash\psi_1\vee\ldots\vee\psi_n$. Let $\chi_1,\ldots,\chi_m$ be any derivation of $\psi_1\vee\ldots\vee\psi_n$ from the premises in $\Gamma$ and let $\phi_1,\ldots,\phi_k$ be a list of all formulas from $\Gamma$ occurring among $\chi_1,\ldots,\chi_m$. Then $\chi_1,\ldots,\chi_m$ also shows the inconsistency (and hence, by Theorem \ref{T:completeness}, the unsatisfiability) of $(\{\phi_1,\ldots,\phi_k\},\{\psi_1,\ldots,\psi_n\})\Subset(\Gamma,\Delta)$. 
\end{proof}

\section{Relations with other logics}\label{S:other}
In the existing literature, one can find several systems which can be viewed as natural companions to $\mathsf{IntCK}$. On the one hand, there are different intensional propositional logics, which either treat conditionals from a viewpoint similar to that of $\mathsf{IntCK}$, or extend $\mathsf{Int}$ with similar additional connectives, or both. On the other hand there is $\mathsf{FOIL}$, the first-order version of $\mathsf{Int}$, which is naturally viewed as a super-system for $\mathsf{IntCK}$, in that $\mathsf{IntCK}$ can be seen as isolating a special subclass of intuitionistic first-order reasoning which is relevant to handling conditionals. In both cases, one can expect that  $\mathsf{IntCK}$ will display some sort of natural relation to each of these logics. This section is devoted to looking into some examples of such relations.
\subsection{Intensional propositional logics}\label{sub:intense}
Due to the great number of systems in this class that can be related to $\mathsf{IntCK}$, we only confine ourselves to mentioning a few prominent examples; and, considering the length of this paper, most of our claims will only be supplied with a rather sketchy proof. We will mostly consider logical systems given by complete Hilbert-style axiomatizations. If $\mathsf{S}$ is such a system and $A_1,\ldots,A_n$ is a finite sequence of axiomatic schemes, we will denote by $\mathsf{S}+\{A_1,\ldots,A_n\}$ the system obtained from $\mathsf{S}$ by adding every instance of $A_1,\ldots,A_n$ and then closing under the applications of the rules of inference assumed in $\mathsf{S}$. In case $n = 1$, we will omit the figure brackets.

The first of the systems that we would like to consider is the basic system $\mathsf{CK}$ of classical conditional logic, introduced in \cite{chellas} and defined over $\mathcal{L}$. One variant of a complete axiomatization for $\mathsf{CK}$ is given by extending of \eqref{E:a0}, \eqref{E:a1}, \eqref{E:a5}, \eqref{E:mp}, \eqref{E:RAbox}, and \eqref{E:RCbox} with the following axiomatic schemes:
\begin{align}
	&\phi\vee\neg\phi\label{E:ax0}\tag{Ax0}\\
	&(\phi\diamondto\psi)\leftrightarrow\neg(\phi\boxto\neg\psi)\label{E:ax1}\tag{Ax1}
\end{align}
The addition of \eqref{E:ax0} to \eqref{E:a0} and \eqref{E:mp} transforms the purely propositional base of the system from $\mathsf{Int}$ to the classical propositional logic $\mathsf{CL}$. It is natural to expect that a similar relation holds between $\mathsf{IntCK}$ and $\mathsf{CK}$ in that the former is the a subsystem of the latter and that $\mathsf{IntCK}$, in its turn, can be transformed into $\mathsf{CK}$ by the addition of \eqref{E:ax0}, thus giving us the intuitionistic counterpart of $\mathsf{CK}$. This is indeed the case, as we will show presently. We prepare the result with a technical lemma:
\begin{lemma}\label{L:CK}
	The following statements are true:
	\begin{enumerate}
		\item Every instance of \eqref{E:a2}--\eqref{E:a4}, \eqref{E:a6}, \eqref{E:RAdiam}, and \eqref{E:RCdiam} as well as all the theorems and derived rules given in Lemma \ref{L:theorems}, are deducible in $\mathsf{CK}$.
		
		\item Every instance of \eqref{E:ax1} is deducible in $ (\mathsf{IntCK} + \eqref{E:ax0})$.
	\end{enumerate}  
\end{lemma}
We sketch the proof in Appendix \ref{A:2}.

The relation between $\mathsf{IntCK}$ and $\mathsf{CK}$ is then analogous to the relation between $\mathsf{Int}$ and $\mathsf{CL}$:
\begin{proposition}\label{P:CK}
	 The following statements are true for every $\phi\in\mathcal{L}$:
	 \begin{enumerate}
	 	\item If $\phi\in\mathsf{IntCK}$, then $\phi\in \mathsf{CK}$.
	 	
	 	\item $\phi\in \mathsf{CK}$ iff $\phi\in (\mathsf{IntCK} + \eqref{E:ax0})$.
	 \end{enumerate}
 \end{proposition}
 \begin{proof}
 	(Part 1) If $\phi\in \mathcal{L}$ and $\phi_1,\ldots,\phi_n = \phi$ is a proof in $\mathsf{IntCK}$, then we can transform it into a proof of $\phi$ in $\mathsf{CK}$ by replacing every occurrence of \eqref{E:a2}--\eqref{E:a4}, \eqref{E:a6}, and every application \eqref{E:RAdiam}, and \eqref{E:RCdiam} by the deductions given in the proof of Lemma \ref{L:CK}.1.
 	
 	(Part 2) If $\phi\in (\mathsf{IntCK} + \eqref{E:ax0})$, then we can argue as in Part 1. The only difference will be possible presence of instances of  \eqref{E:ax0} which do not require any additional work. In the other direction, if $\phi\in \mathsf{CK}$ and $\phi_1,\ldots,\phi_n = \phi$ is a proof in $\mathsf{CK}$, then we can transform it into a proof of $\phi$ in $\mathsf{IntCK}$ by replacing every occurrence of \eqref{E:ax1} by the deduction given in the proof of Lemma \ref{L:CK}.2.
 \end{proof}
Another logic that is very natural to compare with $\mathsf{IntCK}$ is the system of intuitionistic conditional logic $\mathsf{ICK}$ introduced by Y. Weiss in \cite{weiss}. $\mathsf{ICK}$ is defined over the $(\diamondto)$-free fragment of $\mathcal{L}$ which we denote by $\mathcal{L}_\boxto$. One of its complete axiomatizations is obtained by simply omitting \eqref{E:ax0} and \eqref{E:ax1} from $\mathsf{CK}$. One (Chellas-style) variant of semantics\footnote{The semantics given in \cite{weiss} and \cite{weiss-thesis} uses the formula-indexed binary relations but yields the same logic as the semantics we give in this paper.} for $\mathsf{ICK}$ can be given, if we replace conditions \eqref{Cond:1} and \eqref{Cond:2} in Definition \eqref{D:chellas-model} by the following condition to be satisfied for every $X \subseteq W$:
\begin{equation}\label{Cond:w}\tag{cw}
	\leq\circ R_X\subseteq R_X\circ\leq
\end{equation}
We will call the resulting models \textit{Weiss models}. The satisfaction relation used by Y. Weiss (we will be denoting it by $\models_w$) is also different from $\models$ in that the inductive clause for $\diamondto$ is no longer needed and the inductive clause for $\boxto$ is given in the following, more classically-minded\footnote{The Chellas variety of Kripke semantics of $\mathsf{CK}$ is usually defined by adopting this clause over the class of Kripke models that is given as in Definition \ref{D:chellas-model} except that the pre-order $\leq$ is now omitted and conditions \eqref{Cond:1} and \eqref{Cond:2} are no longer imposed. The clause that we used in Section \ref{sub:lands} for $\diamondto$ can be derived in this semantics as a by-product of adopting \eqref{E:ax1} and thus can be also deemed classical. Cf. the semantic clause for $\exists$ which is exactly the same in both classical and intuitionistic first-order logic.} version:
$$
\mathcal{M}, w\models_w \psi \boxto \chi \Leftrightarrow (\forall u \in W)(R_{\|\psi\|_\mathcal{M}}(w, u) \Rightarrow\mathcal{M}, u\models_c \chi)
$$ 
The relations between $\mathsf{ICK}$ and $\mathsf{IntCK}$ can be summarized as follows:
\begin{proposition}\label{P:ICK}
	We have $\mathsf{ICK}\subseteq\mathsf{IntCK}$. However, $\mathsf{IntCK}$ extends $\mathsf{ICK}$ non-conservatively, in that we have $(\neg\neg(\top\boxto\bot)\to (\top\boxto\bot))\in(\mathsf{IntCK}\cap \mathcal{L}_\boxto)\setminus\mathsf{ICK}$.	
\end{proposition}
\begin{proof}
	It is clear that every proof in $\mathsf{ICK}$ is also a proof in $\mathsf{IntCK}$. As for the non-conservativity claim, it is easy to see that $\neg\neg(\top\boxto\bot)\to (\top\boxto\bot)$ is derivable in $\mathsf{IntCK}$:
	\begin{align}
		&\neg\neg(\top\boxto\bot)\leftrightarrow\neg\neg\neg(\top\diamondto\top)\label{E:ick2} &&\text{by \eqref{E:T4}, $\mathsf{Int}$}\\
		&\neg\neg(\top\boxto\bot)\leftrightarrow\neg(\top\diamondto\top)\label{E:ick3} &&\text{by \eqref{E:ick2}, $\mathsf{Int}$}\\
		&((\top\diamondto \top)\to (\top\boxto \bot)) \to (\top\boxto\bot)\label{E:ick6} &&\text{by \eqref{E:a4}, $\mathsf{Int}$}\\
		&\neg\neg(\top\boxto\bot)\to (\top\boxto\bot)&&\text{by \eqref{E:ick3}, \eqref{E:ick6}, $\mathsf{Int}$}
	\end{align} 
	However, if we consider the Weiss model $\mathcal{M} = (W, \leq, R, V)$ where $W: = \{w,v,u\}$, $\leq$ is the reflexive closure of $\{(w,v)\}$, $R := \{(w, W, u)\}$, and $V(p) = \emptyset$ for every $p \in Var$, then we see that $\mathcal{M}, v\models_w \top\boxto\bot$, hence also $\mathcal{M}, w\models_w \neg\neg(\top\boxto\bot)$; however, since $\mathcal{M}, u\not\models_w \bot$ and we have $R_{\|\top\|_\mathcal{M}}(w, u)$, we get that $\mathcal{M}, w\not\models_w \top\boxto\bot$. 
\end{proof}
The question then arises as to how one should interpret this difference between $\mathsf{ICK}$ and $\mathsf{IntCK}\cap \mathcal{L}_\boxto$; is it due to $\mathsf{ICK}$ being incomplete over $\mathcal{L}_\boxto$, or is the reason that $\mathsf{IntCK}$ smuggles in some principles that are not intuitionistically acceptable? The latter answer seems to be favored by the fact that the elimination of double negation is not generally favored by intuitionistic reasoning; moreover, it is clear that the proof of $(\neg\neg(\top\boxto\bot)\to (\top\boxto\bot))$ in $\mathsf{IntCK}$ essentially uses the principles that are only expressible with the help of the additional connective $\diamondto$ and can be therefore seen as, loosely speaking, `impure'. However, the former answer is clearly favored by Theorem \ref{T:foil} of this paper which implies that the standard translation of the questionable formula $(\neg\neg(\top\boxto\bot)\to (\top\boxto\bot))$ is a valid first-order intuitionistic principle. Nevertheless, the weight of the latter argument is somewhat diminished by the fact that the proof of Theorem \ref{T:foil} in this paper depends on classical principles.

Turning once more to $\mathsf{CK}$, it has been shown that it corresponds to the basic modal logic $\mathsf{K}$, which is defined over the language $\mathcal{L}_m$ given by the following BNF:
$$
\phi::= p\mid \top\mid \bot\mid \phi\wedge\phi\mid\phi\vee\phi\mid\phi\to\phi\mid\Box\phi\mid\Diamond\phi.
$$
More precisely, consider the translation $Tr:\mathcal{L}_m\to\mathcal{L}$ defined by the following induction on the construction of $\phi \in \mathcal{L}_m$:
\begin{align*}
	Tr(\psi)&:= \psi&&\psi\in Var\cup\{\top,\bot\}\\
	Tr(\psi\ast\chi)&:= Tr(\psi)\ast Tr(\chi)&&\ast\in \{\wedge, \vee, \to\}\\
	Tr(\Box\psi)&:= \top\boxto Tr(\psi)&&Tr(\Diamond\psi):= \top\diamondto Tr(\psi)
\end{align*}
It follows from the results of \cite{lowe} that $\mathsf{K}$ is embedded into $\mathsf{CK}$ by $Tr$ in the sense that, for every $\phi \in \mathcal{L}_m$, $\phi \in \mathsf{K}$ iff $Tr(\phi) \in \mathsf{CK}$. It is natural to expect that $Tr$ also embeds some basic intuitionistic modal logic into $\mathsf{IntCK}$ and this is indeed the case for the basic intuitionistic modal logic $\mathsf{IK}$ introduced independently in \cite{fischer-servi} and \cite{ps}. 

Just like $\mathsf{K}$, $\mathsf{IK}$ is defined over $\mathcal{L}_m$. One variant of its complete axiomatization can be given by adding to \eqref{E:a0} and \eqref{E:mp} the following additional axiomatic schemes plus a new rule of inference:
\begin{align}
	&\Box(\phi\to\psi)\to(\Box\phi\to\Box\psi)\label{E:aa1}\tag{a1}\\
	&\Box(\phi\to\psi)\to(\Diamond\phi\to\Diamond\psi)\label{E:aa2}\tag{a2}\\
	&\neg\Diamond\bot\label{E:aa3}\tag{a3}\\
	&\Diamond(\phi\vee\psi)\to(\Diamond\phi\vee\Diamond\psi)\label{E:aa4}\tag{a4}\\
	&(\Diamond\phi\to\Box\psi)\to\Box(\phi\to\psi)\label{E:aa5}\tag{a5}\\
	&\text{From }\phi\text{ infer }\Box\phi\label{E:nnec}\tag{nec}	
\end{align}
Again, we prepare the result connecting $\mathsf{IK}$ to $\mathsf{IntCK}$ with a technical lemma:
\begin{lemma}\label{L:IK}
	The following theorems and rules are deducible in $\mathsf{IK}$ for all $\phi,\psi \in \mathcal{L}_m$:
	\begin{align}
		&(\Box\phi\wedge\Box\psi)\leftrightarrow\Box(\phi\wedge\psi)\label{E:t1}\tag{t1}\\
		&(\Diamond\phi\wedge\Box\psi)\to\Diamond(\phi\wedge\psi)\label{E:t2}\tag{t2}\\
		&\Diamond(\phi\vee\psi)\leftrightarrow(\Diamond\phi\vee\Diamond\psi)\label{E:t3}\tag{t3}\\
		&\Box\top\label{E:t4}\tag{t4}\\
		&\text{From }\phi\to\psi\text{ infer }\Box\phi\to\Box\psi\label{E:rl1}\tag{r1}\\
		&\text{From }\phi\to\psi\text{ infer }\Diamond\phi\to\Diamond\psi\label{E:rl2}\tag{r2}\\
		&\text{From }\phi\leftrightarrow\psi\text{ infer }\Box\phi\leftrightarrow\Box\psi\label{E:rl3}\tag{r3}\\
		&\text{From }\phi\leftrightarrow\psi\text{ infer }\Diamond\phi\leftrightarrow\Diamond\psi\label{E:rl4}\tag{r4}
	\end{align}
\end{lemma}
\begin{proof}
	The theorem \eqref{E:t1} and the rule \eqref{E:rl1} can be deduced as in $\mathsf{K}$. To obtain the deduction of \eqref{E:rl2}, one needs to replace the occurrence of  \eqref{E:aa1} in the deduction of \eqref{E:rl1} by the respective occurrence of \eqref{E:aa2}. Rules \eqref{E:rl3} and \eqref{E:rl4} are deduced by applying rules \eqref{E:rl1} and \eqref{E:rl2}, respectively plus the definition of $\leftrightarrow$. The theorem \eqref{E:t4} can be deduced by applying \eqref{E:nnec} to the provable formula $\top$; \eqref{E:t3} is just \eqref{E:aa4}, the other direction follows by applying \eqref{E:rl2} to provable formulas $\phi\to(\phi\vee\psi)$ and $\psi\to(\phi\vee\psi)$. Finally, \eqref{E:t2} can be deduced as follows:
	\begin{align}
		&\Box\psi\to\Box(\phi\to(\phi\wedge\psi))\label{E:ik1}&&\text{\eqref{E:a0}, \eqref{E:mp}, and \eqref{E:rl1}}\\
		&\Box\psi\to(\Diamond\phi\to\Diamond(\phi\wedge\psi))\label{E:ik3}&&\text{\eqref{E:ik1}, \eqref{E:aa2}, \eqref{E:a0}. \eqref{E:mp}}\\
		&(\Diamond\phi\wedge\Box\psi)\to\Diamond(\phi\wedge\psi))&&\text{\eqref{E:ik3}, \eqref{E:a0}. \eqref{E:mp}}
	\end{align}	
\end{proof}
We now claim that:
\begin{proposition}\label{P:IK}
	For every $\phi \in \mathcal{L}_m$, $\phi\in\mathsf{IK}$ iff $Tr(\phi)\in\mathsf{IntCK}$.
\end{proposition}
\begin{proof}
	If $\phi\in\mathsf{IK}$ then let $\phi_1,\ldots,\phi_n = \phi$ be a deduction of $\phi$ in $\mathsf{IK}$. Consider the sequence $Tr(\phi_1),\ldots,Tr(\phi_n) = Tr(\phi)$. The translation $Tr$ leaves intact every instance of \eqref{E:a0} and \eqref{E:mp} and maps every instance of \eqref{E:aa1}, (resp. \eqref{E:aa2}, \eqref{E:aa3}, \eqref{E:aa4}, \eqref{E:aa5}) into an instance of \eqref{E:T1} (resp. \eqref{E:T2}, \eqref{E:a6}, one half of \eqref{E:a3}, \eqref{E:a4}). Similarly, every application of the rule \eqref{E:nnec} is mapped by $Tr$ into an application of \eqref{E:Rnec}. Therefore, one can straightforwardly extend $Tr(\phi_1),\ldots,Tr(\phi_n) = Tr(\phi)$ to a proof of $Tr(\phi)$ in $\mathsf{IntCK}$ by inserting the variants of deductions sketched in the proof of Lemma \ref{L:theorems}.
	
	In the other direction, let $\psi_1,\ldots,\psi_n = Tr(\phi)$ be a deduction of $Tr(\phi)$ in $\mathsf{IntCK}$. Consider the mapping $\overline{Tr}:\mathcal{L}\to\mathcal{L}_m$ defined by induction on the construction of $\phi\in\mathcal{L}$:
	\begin{align*}
		\overline{Tr}(\psi)&:= \psi&&\psi\in Var\cup\{\top,\bot\}\\
		\overline{Tr}(\psi\ast\chi)&:= \overline{Tr}(\psi)\ast \overline{Tr}(\chi)&&\ast\in \{\wedge, \vee, \to\}\\
		\overline{Tr}(\psi\boxto\chi)&:= \Box\overline{Tr}(\chi)&&\overline{Tr}(\psi\diamondto\chi):= \Diamond\overline{Tr}(\chi)
	\end{align*}
The following can be easily proved by induction on the construction of $\phi\in\mathcal{L}_m$:

\textit{Claim}. For every $\phi\in\mathcal{L}_m$, $\overline{Tr}(Tr(\phi)) = \phi$.

Indeed, both basis and every case in the induction step are pretty straightforward. As an example, we consider the case when $\phi = \Diamond\psi$. We have then $\overline{Tr}(Tr(\Diamond\psi)) = \overline{Tr}(\top\diamondto\psi) = \Diamond\psi$. Claim 1 is proven.

Turning back to our proof of $Tr(\phi)$ in $\mathsf{IntCK}$, we consider the sequence of $\mathcal{L}_m$-formulas $\overline{Tr}(\psi_1),\ldots,\overline{Tr}(\psi_n) = \overline{Tr}(Tr(\phi)) = \phi$, where the last equality holds by Claim 1. We observe that the translation given by $\overline{Tr}$ leaves intact every instance of \eqref{E:a0} and \eqref{E:mp}; as for the other axioms and rules, $\overline{Tr}$ maps every instance of \eqref{E:a1}, (resp. \eqref{E:a2}, \eqref{E:a3}, \eqref{E:a4}, \eqref{E:a5}, \eqref{E:a6}) into an instance of \eqref{E:t1} (resp. \eqref{E:t2},  \eqref{E:t3}, \eqref{E:aa5}, \eqref{E:t4}, \eqref{E:aa3}). Similarly, every application of the rule \eqref{E:RCbox} (resp. \eqref{E:RCdiam}) is mapped by $\overline{Tr}$ into an application of the rule \eqref{E:rl3} (resp. \eqref{E:rl4}). Finally, the conclusion of every application of the rule \eqref{E:RAbox} (resp. \eqref{E:RAdiam}) is mapped by $\overline{Tr}$ into a formula of the form $\Box\psi\leftrightarrow\Box\psi$ (resp. $\Diamond\psi\leftrightarrow\Diamond\psi$) which is clearly deducible from \eqref{E:a0} and \eqref{E:mp}. Therefore, one can straightforwardly extend $\overline{Tr}(\psi_1),\ldots,\overline{Tr}(\psi_n)  = \phi$ to a proof of $\phi$ in $\mathsf{IK}$ by inserting the variants of deductions sketched in the proof of Lemma \ref{L:IK}.
\end{proof} 
As a further result of the tight connection between $\mathsf{IK}$ and $\mathsf{IntCK}$ we observe that the countermodels that show in \cite[p. 54--55]{simpson} the mutual non-definability of $\Box$ and $\Diamond$ in $\mathsf{IK}$ can be re-used to show the mutual non-definability of $\boxto$ and $\diamondto$ in $\mathsf{IntCK}$, thus answering the concern expressed in the passage from \cite{ciardelliliu} quoted in the introduction to this paper.

We add that the gap between $\mathsf{IntCK}$ and $\mathsf{ICK}$ is also mirrored at the level of $\mathcal{L}_m$ as the gap between the $\Diamond$-free fragment of $\mathsf{IK}$ and the system $\mathsf{HK}\Box$ introduced in \cite{bozicdosen}, which Y. Weiss cites in \cite[p. 138--139, Footnote 11]{weiss-thesis} as one of the sources for $\mathsf{ICK}$. Indeed, one can easily show both that $\mathsf{ICK}$ stands to $\mathsf{HK}\Box$ in the same relation ascribed to $\mathsf{IntCK}$ and $\mathsf{IK}$ in Proposition \ref{P:IK} above, and that $\neg\neg\Box\bot\to \Box\bot$ is in the  $\Diamond$-free fragment of $\mathsf{IK}$ but outside $\mathsf{HK}\Box$. This is all the more surprising in view of some confusing claims made in the existing literature.\footnote{For example, the author of \cite{simpson} is clearly in favor of $\mathsf{IK}$ as the basic intuitionistic modal logic; yet he says, right after introducing the axiomatization of $\mathsf{HK}\Box$ in Figure 3--3 that: ``The logic of Figure 3--3 is uncontroversially the intuitionistic analogue of $\mathsf{K}$ in the $\Diamond$-free fragment''\cite[p. 56]{simpson}.}

\subsection{First-order intuitionistic logic}\label{sub:foil}
We define the first-order intuitionistic logic $\mathsf{FOIL}$ over the langauge $\mathcal{L}_{fo}$,\footnote{More precisely, we only define the version of $\mathsf{FOIL}$ over a particular vocabulary. But we will never need other versions of $\mathsf{FOIL}$ in this paper, so, for the purposes of the present discourse, we can identify $\mathsf{FOIL}$ with its particular variant that we introduce below.} based on a countable set $Ind$ of individual variables and given by the following BNF:
$$
\phi::= px\mid Rxyz\mid Ox\mid Sx\mid Exy\mid x \equiv y\mid \top\mid \bot\mid\phi\wedge\phi\mid\phi\vee\phi\mid\phi\to\phi\mid\forall x\phi\mid\exists x\phi,
$$
where $p \in Var$ and $x,y,z \in Ind$. The formulas of the form $px$, $Rxyz$, $Ox$, $Sx$, $Exy$, $x\equiv y$, $\top$, and $\bot$, will be called \textit{atomic formulas}, or simply \textit{atoms}. We will continue to use the abbreviations $\leftrightarrow$ and $\neg$; additionally, we will use $(\forall x)_O\phi$ as an abbreviation for $\forall x(Ox\to\phi)$.  
Given a formula $\phi \in \mathcal{L}_{fo}$, we can inductively define its sets of \textit{free} and \textit{bound} variables in a standard way (see, e.g. \cite[p. 64]{vandalen}). These sets, denoted by $FV(\phi)$ and $BV(\phi)$, respectively, are always finite. These notions can be extended to an arbitrary $\Gamma \subseteq \mathcal{L}_{fo}$, although $FV(\Gamma)$ and $BV(\Gamma)$ need not be finite. If $x \in Ind \setminus(FV(\phi)\cup BV(\phi))$, then $x$ is said to be \textit{fresh} for $\phi$. Given any $n \in \omega$ and any $x_1,\ldots, x_n \in Ind$, we will denote by $\mathcal{L}^{\{x_1,\ldots, x_n\}}_{fo}$ the set $\{\phi \in  \mathcal{L}_{fo}\mid FV(\phi) = \{x_1,\ldots, x_n\}\}$. If $\phi \in  \mathcal{L}^\emptyset_{fo}$, then $\phi$ is called a \textit{sentence}. Finally, given a $\phi \in \mathcal{L}_{fo}$, and some $x,y \in Ind$ such that $y$ is fresh for $\phi$, we can define the result $(\phi)^y_x$ of substituting $y$ for $x$ in $\phi$ simply as the result of replacing free occurrences of $x$ in $\phi$ with the occurrences of $y$.

While the most popular semantics for $\mathsf{FOIL}$ is given by Kripke models (see e.g. \cite[Section 5.3]{vandalen}), we will use for this logic a slightly more involved but equivalent semantics based on intuitionistic Kripke sheaves.
The following definition provides the necessary details:
\begin{definition}\label{D:sheaf}
	A \textit{Kripke sheaf} is a structure of the form $\mathfrak{S} = (W, \leq, \mathfrak{A}, \mathbb{H})$ such that:
	\begin{itemize}
		\item $W \neq \emptyset$ is the set of worlds or nodes.
		
		\item $\leq$ is a pre-order on $W$.
		
		\item $\mathfrak{A}$ is a function, returning, for every $w\in W$ a classical first-order model $\mathfrak{A}_w = (A_w, \iota_w)$ over the vocabulary (sometimes called signature) $\Sigma = \{p^1\mid p\in Var\}\cup \{R^3, O^1, S^1, In^2\}$, where $A_w \neq \emptyset$ is the domain and $\iota_w$ is the function assigning every $\mathbb{P}^n\in \Sigma$ some $\iota_w(\mathbb{P}^n)\subseteq (A_w)^n$. 

		\item Finally, $\mathbb{H}$ is a function defined on $\{(w,v)\in W^2\mid w\mathrel{\leq}v\}$, which, for every pair $(w,v)$ in its domain returns a (classical) homomorphism $\mathbb{H}_{wv}:\mathfrak{A}_w\to\mathfrak{A}_v$. This function has to satisfy the following additional conditions:
		\begin{itemize}
			\item For all $w \in W$, $\mathbb{H}_{ww} = id[A_w]$.
			
			\item If $w,v,u \in W$ are such that $w\mathrel{\leq}v\mathrel{\leq}u$, then $\mathbb{H}_{wu} = \mathbb{H}_{wv}\circ\mathbb{H}_{vu}$.
		\end{itemize}
	\end{itemize} 
\end{definition}
We use $\mathfrak{S} = (W, \leq, \mathfrak{A}, \mathbb{H})$ as our standard notation for Kripke sheaves; we will assume that any decorations of $\mathfrak{S}$ transfer to its components.

Given a Kripke sheaf $\mathfrak{S}$, and a $w \in W$, we will call an $(\mathfrak{S},w)$-\textit{variable assignment} any function $f:Ind\to D_w$. In order to determine the truth value of a formula, one needs to supply a Kripke sheaf $\mathfrak{S}$, a node $w \in W$ and an $(\mathfrak{S},w)$-variable assignment $f$. With these data, the satisfaction of a formula $\phi\in\mathcal{L}_{fo}$ is given by the relation $\models_{fo}$ which defined by the following induction on the construction of a formula:
\begin{align*}
	\mathfrak{S}, w\models_{fo}\top[f]\\
	\mathfrak{S}, w\not\models_{fo}\bot[f]\\
	\mathfrak{S}, w\models_{fo}\mathbb{P}x[f] &\Leftrightarrow f(x)\in \iota_w(\mathbb{P}) &&\mathbb{P}\in Var\cup \{S, O\}\\
	\mathfrak{S}, w\models_{fo}Exy[f] &\Leftrightarrow (f(x),f(y))\in \iota_w(In)\\
	\mathfrak{S}, w\models_{fo}(x\equiv y)[f] &\Leftrightarrow f(x) = f(y)\\
	\mathfrak{S}, w\models_{fo}Rxyz[f] &\Leftrightarrow (f(x),f(y),f(z))\in \iota_w(R)\\
	\mathfrak{S}, w\models_{fo}(\psi\wedge\chi)[f] &\Leftrightarrow \mathfrak{S}, w\models_{fo}\psi[f]\,\&\,\mathfrak{S}, w\models_{fo}\chi[f]\\
	\mathfrak{S}, w\models_{fo}(\psi\vee\chi)[f] &\Leftrightarrow \mathfrak{S}, w\models_{fo}\psi[f]\text{ or }\mathfrak{S}, w\models_{fo}\chi[f]\\
	\mathfrak{S}, w\models_{fo}(\psi\to\chi)[f] &\Leftrightarrow (\forall v\mathrel{\geq}w)(\mathfrak{S}, v\models_{fo}\psi[f\circ \mathbb{H}_{wv}]\Rightarrow\mathfrak{S}, v\models_{fo}\chi[f\circ \mathbb{H}_{wv}])\\
	\mathfrak{S}, w\models_{fo}(\exists x\psi)[f] &\Leftrightarrow (\exists a\in D_w)(\mathfrak{S}, w\models_{fo}\psi[f[x/a]])\\
	\mathfrak{S}, w\models_{fo}(\forall x\psi)[f] &\Leftrightarrow (\forall v\mathrel{\geq}w)(\forall a\in D_v)(\mathfrak{S}, v\models_{fo}\psi[(f\circ \mathbb{H}_{wv})[x/a]])
\end{align*}
As usual, it follows from this definition that the truth value of a formula $\phi \in \mathcal{L}_{fo}$ only depends on the values assigned by $f$ to the values of the variables in $FV(\phi)$. We will therefore write $\mathfrak{S}, w\models_{fo}\phi[x_1/a_1,\ldots,x_n/a_n]$ iff $\phi \in \mathcal{L}^{\{x_1,\ldots,x_n\}}_{fo}$ and we have $\mathfrak{S}, w\models_{fo}\phi[f]$ for every (equivalently, any) $(\mathfrak{S},w)$-variable assignment $f$ such that $f(x_i) = a_i$ for every $1 \leq i \leq n$. In particular, we will write $\mathfrak{S}, w\models_{fo}\phi$ iff $\phi \in \mathcal{L}^\emptyset_{fo}$ and we have $\mathfrak{S}, w\models_{fo}\phi[f]$ for every (equivalently, any) $(\mathfrak{S},w)$-variable assignment $f$.

It is easy to see that $\models_{fo}$, just like $\models$ before, can be used to introduce the complete set of semantic notions. More precisely, given a Kripke sheaf $\mathfrak{S}$, and a $w \in W$, and a tuple $(a_1,\ldots, a_n) \in A_w$, we will call the triple $(\mathfrak{S},w,(a_1,\ldots, a_n))$ an $n$-\textit{evaluation point}, and, given a pair $(\Gamma, \Delta) \in \mathcal{P}(\mathcal{L}^{\{x_1,\ldots, x_n\}}_{fo})\times\mathcal{P}(\mathcal{L}^{\{x_1,\ldots, x_n\}}_{fo})$, we say that an evaluation point $(\mathfrak{S},w,(a_1,\ldots, a_n))$ \textit{satisfies}  $(\Gamma, \Delta)$ and write $\mathfrak{S},w \models_c(\Gamma, \Delta)[x_1/a_1,\ldots, x_n/a_n]$ iff we have:
$$
(\forall\phi \in \Gamma)(\mathfrak{S},w\models_{fo} \phi[x_1/a_1,\ldots, x_n/a_n])\,\&\,(\forall\psi \in \Delta)(\mathfrak{S},w\not\models_{fo} \psi[x_1/a_1,\ldots, x_n/a_n]).
$$
We say that  $(\Gamma, \Delta)$ is \textit{first-order-satisfiable} iff some $n$-evaluation point satisfies it, and that $\Delta$ \textit{first-order-follows} from $\Gamma$ (and write $\Gamma \models_{fo} \Delta$) iff $(\Gamma, \Delta)$ is first-order-unsatisfiable. We say that $\Gamma$ is first-order-satisfiable iff $(\Gamma, \emptyset)$ is; and if $(\mathfrak{S},w,(a_1,\ldots, a_n))$ first-order-satisfies $(\Gamma, \emptyset)$, then we simply write $\mathfrak{S},w \models_{fo}\Gamma[x_1/a_1,\ldots, x_n/a_n]$. We say that $\phi \in \mathcal{L}$ is \textit{first-order-valid} iff $\emptyset \models_{fo} \phi$, or, in other words, iff $\models_{fo} \phi$.

$\mathsf{FOIL}$ is known to be strongly complete relative to the semantic of Kripke sheaves; in other words, we have $\Gamma \models_{\mathsf{FOIL}} \Delta$ iff $\Gamma \models_{fo} \Delta$ for all $(\Gamma, \Delta) \in \mathcal{P}(\mathcal{L}^{\{x_1,\ldots, x_n\}}_{fo})\times\mathcal{P}(\mathcal{L}^{\{x_1,\ldots, x_n\}}_{fo})$, see \cite[Section 3.6 ff]{gss} for details.\footnote{A lightened version is given in \cite[Cor. 5.3.16]{vandalen}, where the semantics of Kripke sheaves is referred to as `modified Kripke semantics'.} However, in this paper we will be mainly interested in things that first-order-follow from a specific subset $Th\subseteq \mathcal{L}^\emptyset_{fo}$. This subset includes all and only the sentences that follow below:
\begin{align}
	&\forall x(Sx\vee Ox)\label{E:th1}\tag{Th1}\\
	&\forall x\neg(Sx\wedge Ox)\label{E:th2}\tag{Th2}\\
	&\forall x(px\to Ox) &&p\in Var\label{E:th3}\tag{Th3}\\
	&\forall x\forall y(Exy\to (Ox\wedge Sy))\label{E:th4}\tag{Th4}\\
	&\forall x\forall y\forall z(Rxyz\to (Ox\wedge Sy\wedge Oz))\label{E:th5}\tag{Th5}\\
	&\exists x(Sx\wedge \forall y(Eyx\leftrightarrow py)) &&p\in Var\label{E:th6}\tag{Th6}\\
	&\exists x(Sx\wedge (\forall y)_O(Eyx))\label{E:th7}\tag{Th7}\\
	&\exists x(Sx\wedge \forall y\neg Eyx)\label{E:th8}\tag{Th8}\\
	&\forall x\forall y((Sx\wedge Sy)\to\exists z(Sz\wedge (\forall w)_O(Ewz\leftrightarrow (Ewx\ast Ewy)))) &&\ast\in \{\wedge, \vee, \to\}\label{E:th9}\tag{Th9}\\
	&\forall x\forall y((Sx\wedge Sy)\to\exists z(Sz\wedge (\forall w)_O(Ewz\leftrightarrow \forall u(Rwxu\to Euy))))\label{E:th10}\tag{Th10}\\
	&\forall x\forall y((Sx\wedge Sy)\to\exists z(Sz\wedge (\forall w)_O(Ewz\leftrightarrow \exists u(Rwxu\wedge Euy))))\label{E:th11}\tag{Th11}\\
	&\forall x\forall y(Sx\wedge Sy \wedge (\forall z)_O(Ezx\leftrightarrow Ezy) \to x \equiv y)\label{E:th12}\tag{Th12}
\end{align}
It is easy to notice that $Th$ encodes a two-sorted universe, where the sorts are represented by the unary predicates $O$ (for `objects') and $S$ (for `sets'). The predicate $E$ then represents the elementhood. The principles \eqref{E:th6}--\eqref{E:th12} ensure that the sets assigned to the formulas of $\mathcal{L}$ as their extensions by the classical Kripke semantics of $\mathsf{CK}$, are all present in the domain of any model of $Th$.

These clauses can be rather straightforwardly formalized in $\mathsf{FOIL}$ also in the form of the following \textit{standard translation} of $\mathcal{L}$ into the two-sorted first-order language. Given an $x \in Ind$ and a $\phi \in \mathcal{L}$, the formula $ST_x(\phi)\in \mathcal{L}_{fo}$ is called the standard $x$-translation of $\phi$ and is defined by the following induction on the construction of $\phi$ (where $x, y, z,w \in Ind$ are pairwise distinct):
\begin{align*}
	ST_x(p) &:= px \qquad\qquad\qquad\qquad\qquad\qquad\qquad\qquad\qquad p \in Var\\
	ST_x(\phi) &:= \phi\qquad\qquad\qquad\qquad\qquad\qquad\qquad\qquad\qquad \phi \in \{\top, \bot\}\\
	ST_x(\psi\ast\chi) &:= ST_x(\psi)\ast ST_x(\chi)\qquad\qquad\qquad\qquad\qquad\qquad\ast\in \{\wedge,\vee,\to\}\\
	ST_x(\psi\boxto\chi) &:= \exists y(Sy \wedge (\forall z)_O (Ezy\leftrightarrow ST_z(\psi))\wedge\forall w(Rxyw\to ST_w(\chi)))\\
	ST_x(\psi\diamondto\chi) &:= \exists y(Sy \wedge (\forall z)_O (Ezy\leftrightarrow ST_z(\psi))\wedge \exists w(Rxyw\wedge ST_w(\chi)))
\end{align*}
Thus the standard translation simply encodes the (Segerberg version of) Kripke semantics for $\mathsf{CK}$ under thte assumption that the underlying logic of the meta-language is intuitionistic rather than classical. The following lemma is then easily proved by induction on the construction of $\phi\in\mathcal{L}$:
\begin{lemma}\label{L:st}
	For every $x\in Ind$ and every $\phi\in\mathcal{L}$, $ST_x(\phi)\subseteq \mathcal{L}^x_{fo}$.
\end{lemma}
The next couple of lemmas are more tedious, but still rather straightforward to prove:
\begin{lemma}\label{L:th-existence-1}
	For all $\phi, \psi \in \mathcal{L}$ and for all pairwise distinct $x, y,z,w \in Ind$, the following statements hold:
	\begin{align*}
1.\,Th\models_{fo}\forall x(ST_x(\psi\boxto\chi) &\to \forall y(Sy \wedge( \forall z)_O(Ezy\leftrightarrow ST_z(\psi))\to\forall w(Rxyw\to ST_w(\chi)))).
	\end{align*}
\begin{align*}
	2.\,Th\models_{fo}\forall x(ST_x(\psi\diamondto\chi) &\to \forall y(Sy \wedge (\forall z)_O(Ezy\leftrightarrow ST_z(\psi))\to\exists w(Rxyw\wedge ST_w(\chi)))).
\end{align*}
\end{lemma}
\begin{lemma}\label{L:th-existence-2}
	For every $\phi\in \mathcal{L}$ and for all distinct $x,y \in Ind$ the following holds:
	$$
	Th\models_{fo}\exists x(Sx\wedge(\forall y)_O(Eyx\leftrightarrow ST_y(\phi))).
	$$
\end{lemma} 
These lemmas provide a stepping stone for our first result on the relations between $Th$ and $\mathsf{IntCK}$:
\begin{proposition}\label{P:easy}
	For every $\phi\in\mathcal{L}$ and every $x \in Ind$, if $\phi\in\mathsf{IntCK}$ then $Th\models_{fo}\forall x(ST_x(\phi))$.
\end{proposition}
The proofs of Lemmas \ref{L:th-existence-1}, \ref{L:th-existence-2}, and of Proposition \ref{P:easy} can be found in Appendix \ref{A:3}.

Using the compactness of $\mathsf{FOIL}$, we are now in a position prove one direction of the main result for the present section:
\begin{corollary}\label{C:easy}
	For all $\Gamma, \Delta \subseteq \mathcal{L}$ and for every $x \in Ind$, if $\Gamma \models_{\mathsf{IntCK}} \Delta$, then $Th, \{ST_x(\phi)\mid \phi \in \Gamma\}\models_{fo}\{ST_x(\psi)\mid \psi \in \Delta\}$.
\end{corollary}
\begin{proof}
	If $\Gamma \models_{\mathsf{IntCK}} \Delta$ then, by Theorem \ref{T:completeness}, $(\Gamma, \Delta)$ must be unsatisfiable, whence, by Corollary \ref{C:compactness}, there must exist some $\Gamma'\Subset\Gamma$ and some $\Delta'\Subset\Delta$ such that $(\Gamma', \Delta')$ is unsatisfiable. But then, again, by Theorem \ref{T:completeness} and Lemma \ref{L:alt-consistency}, we must have $\bigwedge\Gamma'\to\bigvee\Delta'\in \mathsf{IntCK}$. Proposition \ref{P:easy} now implies that $Th\models_{fo} \forall xST_x(\bigwedge\Gamma'\to\bigvee\Delta')$, whence clearly $Th\models_{fo} \forall x(\bigwedge\{ST_x(\phi)\mid \phi \in \Gamma'\}\to\bigvee\{ST_x(\psi)\mid \psi \in \Delta\})$, and, furthermore:
	$Th, \{ST_x(\phi)\mid \phi \in \Gamma'\}\models_{fo}\{ST_x(\psi)\mid \psi \in \Delta'\}$.
	
	But then, trivially, also $Th, \{ST_x(\phi)\mid \phi \in \Gamma\}\models_{fo}\{ST_x(\psi)\mid \psi \in \Delta\}$ must hold. 
\end{proof} 
\begin{example}\label{example}
	To illustrate the import of Proposition \ref{P:easy}, let us consider the intuitionistic meaning of the formula $\phi:= \neg\neg(\top\boxto\bot)\to(\top\boxto\bot)$, whose status as a part of basic intuitionistic logic of conditional is, as we saw in Section \ref{sub:intense}, disputed between $\mathsf{ICK}$ and $\mathsf{IntCK}$. Setting $\psi: = Sy\wedge(\forall z)_O(Ezy)$, we see that we must have $
	ST_x(\phi) = \neg\neg\exists y(\psi\wedge\forall w\neg Rxyw)\to\exists y(\psi\wedge\forall w\neg Rxyw)$.
	By \eqref{E:th7}, we know that $Th \models_{fo}\exists y\psi$, so the whole formula is reduced, modulo $Th$, to the following theorem of $\mathsf{FOIL}$: $
	\neg\neg\forall w\neg Rxyw\to \forall w\neg Rxyw$.
	It is clear now that the case of double negation elimination claimed in $\phi$ is intuitionistically acceptable and that $\phi$ indeed must be accepted as a theorem of basic intuitionistic conditional logic.
\end{example}
We now want to prove the converse to Corollary \ref{C:easy}; to this end, we want to throw a more attentive look at the canonical model $\mathcal{M}_c$ given in Definition \ref{D:canonical-model}. We adopt the following further definitions:
\begin{definition}\label{D:standard-sequence}
	For any $n \in \omega$, a sequence $\alpha = ((\Gamma_0,\Delta_0),\phi_1,\ldots,\phi_n, (\Gamma_n, \Delta_n))$ is called a $(\Gamma_0,\Delta_0)$-standard sequence of length $n + 1$ iff $(\Gamma_0,\Delta_0),\ldots,(\Gamma_n, \Delta_n)\in W_c$, $\phi_1,\ldots,\phi_n \in \mathcal{L}$, and we have:
	$$
	(\forall i < n)(R_c((\Gamma_i,\Delta_i),\|\phi_{i + 1}\|_{\mathcal{M}_c}, (\Gamma_{i+1}, \Delta_{i+1}))).
	$$
Given a $(\Gamma, \Delta)\in W_c$, the set of all $(\Gamma,\Delta)$-standard sequences will be denoted by $Seq(\Gamma,\Delta)$. The set $Seq$ of all standard sequences is then given by $
Seq := \bigcup\{Seq(\Gamma,\Delta)\mid (\Gamma, \Delta)\in W_c\}$.

Finally, given a $\beta = ((\Xi_0,\Theta_0),\psi_1,\ldots,\psi_m, (\Xi_m, \Theta_m))\in Seq$, we say that $\beta$ extends $\alpha$ and will write $
\alpha \mathrel{\prec}\beta$ 
iff (1) $m = n$, (2) $\phi_1 = \psi_1,\ldots, \phi_n = \psi_n (= \psi_m)$, and

\noindent(3) $(\Gamma_0,\Delta_0)\leq_c (\Xi_0,\Theta_0),\ldots,  (\Gamma_n,\Delta_n)\leq_c (\Xi_n,\Theta_n)(=  (\Xi_m,\Theta_m))$. 
\end{definition}
We will need the following facts about standard sequences:
\begin{lemma}\label{L:standard-sequence}
The following statements hold:
\begin{enumerate}
	\item If $\alpha, \beta \in Seq$ are such that $\alpha\mathrel{\prec}\beta$, and $\phi \in \mathcal{L}$, $(\Gamma, \Delta) \in W_c$ are such that $\alpha' = \alpha^\frown(\phi, (\Gamma, \Delta)) \in Seq$, then there exists a $\beta' = \beta^\frown(\phi, (\Xi, \Theta)) \in Seq$ and $\alpha'\mathrel{\prec}\beta'$. 
	
	\item For all $(\Gamma,\Delta), (\Xi,\Theta) \in W_c$ such that $(\Gamma,\Delta)\leq_c (\Xi,\Theta)$, and every $\alpha \in Seq(\Gamma,\Delta)$, there exists a $\beta \in Seq(\Xi, \Theta)$ such that $\alpha\mathrel{\prec}\beta$.
	\item For all $(\Gamma,\Delta), (\Xi,\Theta) \in W_c$ and $\alpha \in Seq(\Gamma,\Delta)$ such that $end(\alpha)\leq_c(\Xi,\Theta)$, there exists a $\beta \in Seq$ such that $\alpha\mathrel{\prec}\beta$ and $end(\beta) = (\Xi,\Theta)$. 
\end{enumerate} 	
\end{lemma}
\begin{proof}
	(Part 1) Assume the hypothesis. Since $\alpha\mathrel{\prec}\beta$, we must have $end(\alpha)\leq_c end(\beta)$, and, since $\alpha' = \alpha^\frown(\phi, (\Gamma, \Delta)) \in Seq$, we must have $R_c(end(\alpha),\|\phi\|_{\mathcal{M}_c}, (\Gamma, \Delta))$. Since $\mathcal{M}_c$ satisfies condition \eqref{Cond:1} of Definition \ref{D:chellas-model}, there must exist some $(\Xi, \Theta)\in W_c$ such that both $(\Gamma, \Delta)\mathrel{\leq_c}(\Xi, \Theta)$ and $R_c(end(\beta),\|\phi\|_{\mathcal{M}_c}, (\Xi, \Theta))$. The latter means that we have both $\beta' = \beta^\frown(\phi, (\Xi, \Theta)) \in Seq$ and $\alpha'\mathrel{\prec}\beta'$.
	
	(Part 2) By induction on the length of $\alpha$. If $\alpha$ has length $1$, then we can set $\beta:= (\Xi,\Theta)$. If $\alpha$ has length $k + 1$ for some $1 \leq k < \omega$, then we apply IH and Part 1.
	
	(Part 3) Again, we proceed by induction on the length of $\alpha$. If $\alpha$ has length $1$, then we can set $\beta:= (\Xi,\Theta)$. If $\alpha$ has length $k + 1$ for some $1 \leq k < \omega$, then for some $\alpha' \in Seq$ of length $k$ and some $\phi \in \mathcal{L}$, we must have $\alpha = ((\Gamma, \Delta), \phi)^\frown(\alpha')$, with $end(\alpha) = end(\alpha')$. Applying now IH to $\alpha'$, we find some $\beta' \in Seq$ such that $end(\beta') = (\Xi,\Theta)$ and $\alpha'\mathrel{\prec}\beta'$. But then, in particular, we must have $init(\alpha') \leq_c init(\beta')$. Moreover, since $\alpha = ((\Gamma, \Delta), \phi)^\frown(\alpha') \in Seq$, we must also have $R_c((\Gamma, \Delta),\|\phi\|_{\mathcal{M}_c}, init(\alpha'))$. But then, since $\mathcal{M}_c$ satisfies condition \eqref{Cond:2} of Definition \ref{D:chellas-model}, there must exist some $(\Gamma', \Delta')\in W_c$ such that both $(\Gamma, \Delta)\mathrel{\leq_c}(\Gamma', \Delta')$ and $R_c((\Gamma', \Delta'),\|\phi\|_{\mathcal{M}_c},init(\beta'))$. But then, for $\beta := ((\Gamma', \Delta'), \phi)^\frown(\alpha')$ we must have $\beta \in Seq$, $end(\beta) = end(\beta') = (\Xi,\Theta)$, and  $\alpha\mathrel{\prec}\beta$.  
\end{proof}
Given any $(\Gamma, \Delta), (\Xi, \Theta)\in W_c$ such that $(\Gamma,\Delta)\leq_c (\Xi,\Theta)$, a function $f:Seq(\Gamma,\Delta)\to Seq(\Xi,\Theta)$ is called a \textit{local} $((\Gamma,\Delta),(\Xi,\Theta))$-\textit{choice function} iff $(\forall \alpha \in Seq(\Gamma,\Delta))(\alpha\mathrel{\prec}f(\alpha))$. The set of all local $((\Gamma,\Delta),(\Xi,\Theta))$-choice functions will be denoted by $\mathfrak{F}((\Gamma,\Delta),(\Xi,\Theta))$. The following Lemma sums up some facts about local choice functions:
\begin{lemma}\label{L:choice-functions}
	Let  $(\Gamma, \Delta) \in W_c$. Then the following statements hold:
	\begin{enumerate}
		\item For every $(\Xi, \Theta)\in W_c$ such that $(\Gamma,\Delta)\leq_c (\Xi,\Theta)$, we have $\mathfrak{F}((\Gamma,\Delta),(\Xi,\Theta))\neq \emptyset$.
		
		\item For every $\alpha \in Seq(\Gamma,\Delta)$ and every $(\Xi, \Theta)\in W_c$ such that $end(\alpha) \leq_c (\Xi, \Theta)$, there exist a $\beta \in Seq$ such that $end(\beta) = (\Xi, \Theta)$, and an $f\in\mathfrak{F}((\Gamma,\Delta),init(\beta))$ such that $f(\alpha) = \beta$.
		
		\item For every $(\Xi, \Theta)\in W_c$, $id[Seq(\Xi,\Theta)]\in \mathfrak{F}((\Xi,\Theta),(\Xi,\Theta))$.
		
		\item Given a $n \in \omega$, any $(\Gamma_0,\Delta_0),\ldots,(\Gamma_n, \Delta_n)\in W_c$ such that $(\Gamma_0,\Delta_0)\leq_c\ldots\leq_c(\Gamma_n, \Delta_n)$, and any $f_1,\ldots,f_n$ such that for every $i < n$ we have $f_{i + 1}\in \mathfrak{F}((\Gamma_i, \Delta_i),(\Gamma_{i + 1}, \Delta_{i + 1}))$, we also have that $f_1\circ\ldots\circ f_n \in \mathfrak{F}((\Gamma_0, \Delta_0),(\Gamma_{n}, \Delta_{n}))$.
	\end{enumerate} 
\end{lemma}
\begin{proof}
	(Part 1) By Lemma  \ref{L:standard-sequence}.2 and Axiom of Choice.
	
	(Part 2) Assume the hypothesis. By Lemma  \ref{L:standard-sequence}.3, we can find a $\beta\in Seq$ such that both $\alpha\mathrel{\prec}\beta$ and $end(\beta) = (\Xi, \Theta)$ are satisfied. Trivially, we must also have $\beta \in Seq(init(\beta))$. Assume, wlog, that $\alpha = ((\Gamma_0,\Delta_0),\phi_1,\ldots,\phi_n, (\Gamma_n, \Delta_n))$ and that $\beta = ((\Xi_0,\Theta_0),\phi_1,\ldots,\phi_n, (\Xi_n, \Theta_n))$. We now set $f((\Gamma_0,\Delta_0),\phi_1,\ldots,\phi_i, (\Gamma_i, \Delta_i)):= ((\Xi_0,\Theta_0),\phi_1,\ldots,\phi_i, (\Xi_i, \Theta_i))$ for every $i < n$ and, proceeding by induction on the length of a standard sequence, extend this partial function to other elements of $Seq(\Gamma,\Delta)$ in virtue of Lemma \ref{L:standard-sequence}.1. The resulting function $f$ clearly has the desired properties.
	
	Part 3 is trivial, and an easy induction on $n \in \omega$ also yields us Part 4.
\end{proof}
 However, $((\Gamma,\Delta),(\Xi,\Theta))$-choice functions will mostly interest us as restrictions of \textit{global choice functions}. More precisely, $F: Seq \to Seq$ is a global choice function iff for every $(\Gamma, \Delta) \in W_c$ there exists a $(\Xi, \Theta)\in W_c$ such that $(\Gamma,\Delta)\leq_c (\Xi,\Theta)$ and $F\upharpoonright Seq(\Gamma,\Delta) \in \mathfrak{F}((\Gamma,\Delta),(\Xi,\Theta))$. The set of all global choice functions will be denoted by $\mathfrak{G}$. It is easy to see that an analogue of Lemma \ref{L:choice-functions} can be proven for global choice functions:
\begin{lemma}\label{L:global-functions}
	Let  $(\Gamma, \Delta) \in W_c$. Then the following statements hold:
	\begin{enumerate}
		\item For every $(\Xi, \Theta)\in W_c$ such that $(\Gamma,\Delta)\leq_c (\Xi,\Theta)$ and for every $f \in \mathfrak{F}((\Gamma,\Delta),(\Xi,\Theta))$, there exists an $F \in \mathfrak{G}$ such that $F\upharpoonright Seq(\Gamma,\Delta) = f$.
		
		\item For every $\alpha \in Seq(\Gamma,\Delta)$ and every $(\Xi, \Theta)\in W_c$ such that $end(\alpha) \leq_c (\Xi, \Theta)$, there exists a $\beta \in Seq$ such that $end(\beta) = (\Xi, \Theta)$, and an $F\in\mathfrak{G}$ such that $F(\alpha) = \beta$.
		
		\item $id[Seq]\in \mathfrak{G}$.
		
		\item Given any $n \in \omega$ and any $F_1,\ldots,F_n \in \mathfrak{G}$ we also have that $F_1\circ\ldots\circ F_n \in \mathfrak{G}$.
		
		\item For every $\alpha \in Seq$ and every $F \in \mathfrak{G}$, we have $\alpha\mathrel{\prec}F(\alpha)$.
	\end{enumerate} 
\end{lemma}
\begin{proof}
	(Part 1) Note that we have $Seq(\Gamma_0,\Delta_0) \cap Seq(\Gamma_1,\Delta_1) = \emptyset$ whenever $(\Gamma_0,\Delta_0), (\Gamma_1,\Delta_1) \in W_c$ are such that $(\Gamma_0,\Delta_0)\neq (\Gamma_1,\Delta_1)$. Therefore, we can define the global choice function in question by $F:= (\bigcup\{Id[Seq(\Gamma_0,\Delta_0)]\mid (\Gamma_0,\Delta_0)\neq (\Gamma,\Delta)\}) \cup f$.
	Lemma \ref{L:choice-functions}.3 then implies that $F\in\mathfrak{G}$.
	
	(Part 2) Assume the hypothesis. By Lemma  \ref{L:choice-functions}.2, we can choose a $\beta \in Seq$ such that $end(\beta) = (\Xi, \Theta)$, and an $f\in\mathfrak{F}((\Gamma,\Delta),init(\beta))$ such that $f(\alpha) = \beta$. By Part 1, we can find an $F\in\mathfrak{G}$ such that $F\upharpoonright Seq(\Gamma,\Delta) = f$ and thus also $F(\alpha) = f(\alpha) = \beta$.  Parts 3 and 4 are, again, straightforward.
	
	As for Part 5, note that we must have  both $\alpha \in Seq(init(\alpha))$ and, for an appropriate $(\Gamma, \Delta) \in W_c$, that $F\upharpoonright Seq(init(\alpha)) \in \mathfrak{F}(init(\alpha), (\Gamma, \Delta))$. But then we must have 
	\noindent$\alpha\mathrel{\prec}F\upharpoonright Seq(init(\alpha)) = F(\alpha)$ by definition of a local choice function.
\end{proof} 
 The global choice functions are the basis for another type of sequences, that, along with the standard sequences, is necessary for the main model-theoretic construction of the present section. We will call them \textit{global sequences}. A global sequence is any sequence of the form $(F_1,\ldots, F_n) \in \mathfrak{G}^n$ where $n \in \omega$ (thus $\Lambda$ is also a global sequence with $n = 0$). Given two global sequences $(F_1,\ldots, F_k)$ and $(G_1,\ldots, G_m)$, we say that $(G_1,\ldots, G_m)$ \textit{extends} $(F_1,\ldots, F_k)$ and write $(F_1,\ldots, F_k)\mathrel{\sqsubseteq}(G_1,\ldots, G_m)$ iff $k \leq m$ and $F_1 = G_1,\ldots, F_k = G_k$. Furthermore, we will denote by $Glob$ the set $\bigcup_{n \in \omega}\mathfrak{G}^n$, that is to say, the set of all global sequences.
 
The final item in this series of preliminary model-theoretic constructions is a certain equivalence relation on $\mathcal{L}$. Namely, given any $\phi,\psi \in \mathcal{L}$ and any $(\Gamma, \Delta)\in W_c$, we define that:
$$
\phi\sim\psi :\Leftrightarrow (\phi \leftrightarrow \psi \in \mathsf{IntCK}).
$$
For any $\phi \in \mathcal{L}$, we will denote its equivalence class relative to $\sim$ by $[\phi]_\sim$.

We now proceed to define a particular sheaf $\mathfrak{S}_c$ which can be seen as induced by the model $\mathcal{M}_c$ of Definition \ref{D:canonical-model}.
\begin{definition}\label{D:canonical-sheaf}
	We set $\mathfrak{S}_c:= (Glob, \sqsubseteq, \mathcal{A}, \mathbb{F})$, where:
	\begin{itemize}
		\item For every $\bar{F} \in Glob$, $\mathcal{A}_{\bar{F}} = \mathcal{A} = (\mathbb{A}, \iota)$, i.e. every global sequence gets assigned one and the same classical model $\mathcal{A}$. As for the components of $\mathcal{A}$, we set:
		\begin{itemize}
			\item $\mathbb{A} := Seq \cup \{[\phi]_{\sim}\mid \phi \in \mathcal{L}\}$. 
			
			\item $\iota(p) := \{\beta \in Seq\mid p\in \pi^1(end(\beta))\}$, for every $p \in Var$.
			
			\item $\iota(O) := Seq$.
			
			\item $\iota(S) := \{[\phi]_{\sim}\mid \phi \in \mathcal{L}\}$.
			
			\item $\iota(E):= \{(\beta, [\phi]_{\sim})\in \iota(O)\times\iota(S)\mid \phi \in \pi^1(end(\beta))\}$.
			
			\item $\iota(R):= \{(\beta, [\phi]_{\sim}, \gamma)\in \iota(O)\times\iota(S)\times\iota(O)\mid (\exists (\Xi, \Theta) \in W_c)(\exists \psi \in [\phi]_{\sim})(\gamma = \beta^\frown(\psi, (\Xi, \Theta)))\}$.
		\end{itemize} 
		\item For  any $\bar{F}, \bar{G} \in Glob$ such that, for some $k \leq m$ we have  $\bar{F} = (F_1,\ldots, F_k)$ and  $\bar{G} = (F_1,\ldots, F_m)$ we have $\mathbb{F}_{\bar{F}\bar{G}}:\mathbb{A}\to \mathbb{A}$, where we set:
		$$
		\mathbb{F}_{\bar{F}\bar{G}}(\gamma):= \begin{cases}
			(F_{k + 1}\circ\ldots\circ F_n)(\gamma),\text{ if }\gamma \in  Seq;\\
			\gamma,\text{ otherwise. }
		\end{cases}
		$$
	\end{itemize}
\end{definition}
We show that  $\mathfrak{S}_{(\Gamma, \Delta)}$ is, indeed, a structure that can be used for discussing $Th$ and its relation to $\mathsf{IntCK}$. Again, we begin by establishing another technical fact:
\begin{lemma}\label{L:technical}
	For all $\bar{F}, \bar{G}\in  Glob$ such that $\bar{F}\mathrel{\sqsubseteq}\bar{G}$, it is true that:
	\begin{enumerate}
		\item $\mathbb{F}_{\bar{F}\bar{G}}\upharpoonright Seq \in \mathfrak{G}$.
		
		\item For every $\alpha \in Seq$ we have $\alpha\mathrel{\prec}\mathbb{F}_{\bar{F}\bar{G}}(\alpha)$; in particular, we have $end(\alpha)\leq_c end(\mathbb{F}_{\bar{F}\bar{G}}(\alpha))$, or, equivalently, $\pi^1(end(\alpha))\subseteq \pi^1(end(\mathbb{F}_{\bar{F}\bar{G}}(\alpha)))$.
	\end{enumerate} 
\end{lemma}
\begin{proof}
	Assume the hypothesis; we may also assume, wlog, that, for some $k \leq m < \omega$, $\bar{F}$ and $\bar{G}$ are given in the following form:
	\begin{align}
		\bar{F} &= (F_1,\ldots, F_k)\label{E:barf}\tag{def-$\bar{F}$}\\
		\bar{G} &= (F_1,\ldots, F_m)\label{E:barg}\tag{def-$\bar{G}$}
	\end{align}
	In this case, we also get the following representation for $\mathbb{F}_{\bar{F}\bar{G}}$:
	\begin{align}
		\mathbb{F}_{\bar{F}\bar{G}}\upharpoonright Seq = id[\mathbb{A}]\circ F_{k + 1}\circ\ldots\circ F_m\label{E:fab1}\tag{def1-$\mathbb{F}_{\bar{F}\bar{G}}$}\\
		\mathbb{F}_{\bar{F}\bar{G}}\upharpoonright \{[\phi]_{\sim}\mid \phi \in \mathcal{L}\} = id[\{[\phi]_{\sim}\mid \phi \in \mathcal{L}\}]\label{E:fab2}\tag{def2-$\mathbb{F}_{\bar{F}\bar{G}}$}
	\end{align} 
Part 1 now easily follows from \eqref{E:fab1} and Lemma \ref{L:global-functions}.4.

As for Part 2, we observe that, if $\alpha \in Seq$, then
	$$
	\mathbb{F}_{\bar{F}\bar{G}}(\alpha) =  (id[\mathbb{A}]\circ F_{k + 1}\circ\ldots\circ F_m)(\alpha)\in Seq.
	$$
	Now Part 1 and Lemma \ref{L:global-functions}.5 together imply that $\alpha\mathrel{\prec}\mathbb{F}_{\bar{F}\bar{G}}(\alpha)$. By Definition \ref{D:standard-sequence}, this means that $end(\alpha)\leq_c end(\mathbb{F}_{\bar{F}\bar{G}}(\alpha))$, or, equivalently, that $\pi^1(end(\alpha))\subseteq \pi^1(end(\mathbb{F}_{\bar{F}\bar{G}}(\alpha)))$. 	
\end{proof}

\begin{lemma}\label{L:canonical-sheaf}
	$\mathfrak{S}_c$ is a Kripke sheaf.
\end{lemma}
\begin{proof}
	It is clear that $Glob$ is non-empty and that $\sqsubseteq$ defines a pre-order on $Glob$. It is also clear that $\mathcal{A} = (\mathbb{A}, \iota)$, as given in Definition \ref{D:canonical-sheaf}, makes up a classical model, and that, for any given $\bar{F}, \bar{G}\in  Glob$ such that $\bar{F}\mathrel{\sqsubseteq}\bar{G}$, we have $\mathbb{F}_{\bar{F}\bar{G}}:\mathbb{A}\to \mathbb{A}$.
	
	As for the conditions imposed by Definition \ref{D:sheaf}  on the functions of the form $\mathbb{F}_{\bar{F}\bar{G}}$, it is clear from Definition \ref{D:canonical-sheaf} and from our convention on the superpositions of empty families of functions that (1) for any $\bar{F}\in  Glob$ we will have in $\mathbb{F}_{\bar{F}\bar{F}} = id[\mathbb{A}]$ and that (2) if $\bar{F}, \bar{G}, \bar{H}\in  Glob$ are such that $\bar{F}\mathrel{\sqsubseteq}\bar{G}\mathrel{\sqsubseteq}\bar{H}$, then $\mathbb{F}_{\bar{F}\bar{G}}\circ\mathbb{F}_{\bar{G}\bar{H}} = \mathbb{F}_{\bar{F}\bar{H}}$.

It remains to establish that, for each pair $\bar{F}, \bar{G}\in  Glob$ such that  $\bar{F}\mathrel{\sqsubseteq}\bar{G}$, the function  $\mathbb{F}_{\bar{F}\bar{G}}$ is a (classical) homomorphism from $\mathcal{A}$ to itself. The latter claim boils down to showing that the extension of every predicate $\mathbb{P} \in \Sigma$ is preserved by $\mathbb{F}_{\bar{F}\bar{G}}$. In doing so, we will assume that, for some appropriate $k \leq m < \omega$, $\bar{F}$, $\bar{G}$, and $\mathbb{F}_{\bar{F}\bar{G}}$ are given in a form that satisfies \eqref{E:barf}, \eqref{E:barg}, \eqref{E:fab1}, and \eqref{E:fab2}. The following cases have to be considered:

\textit{Case 1}. $\mathbb{P} \in \{O, S\}$. Trivial by \eqref{E:fab1} and \eqref{E:fab2}.

\textit{Case 2}. $\mathbb{P} = p \in Var$. If $\alpha \in \mathbb{A}$ is such that $\alpha \in \iota(p)$, then, by Definition \ref{D:canonical-sheaf} and Lemma \ref{L:technical}, we must have all of the following: (1) $\alpha, \mathbb{F}_{\bar{F}\bar{G}}(\alpha) \in Seq$, (2) $\pi^1(end(\alpha))\subseteq \pi^1(end(\mathbb{F}_{\bar{F}\bar{G}}(\alpha)))$, and (3) $p\in \pi^1(end(\alpha))$. But then clearly also $p \in \pi^1(end(\mathbb{F}_{\bar{F}\bar{G}}(\alpha)))$, whence, further, $\mathbb{F}_{\bar{F}\bar{G}}(\alpha) \in \iota(p)$, as desired.

\textit{Case 3}. $\mathbb{P} = E$. If $\alpha, \beta \in \mathbb{A}$ are such that $(\alpha, \beta) \in \iota(E)$, then, arguing as in Case 2, we must have: (1) $\alpha, \mathbb{F}_{\bar{F}\bar{G}}(\alpha) \in Seq$, (2) $\beta \in \iota(S)$, in other words, $\beta = [\phi]_{\sim}$ for some $\phi \in \mathcal{L}$, (3) $\pi^1(end(\alpha))\subseteq \pi^1(end(\mathbb{F}_{\bar{F}\bar{G}}(\alpha)))$, and (4) $\phi \in \pi^1(end(\alpha))\subseteq \pi^1(end(\mathbb{F}_{\bar{F}\bar{G}}(\alpha)))$, so that we also have, by \eqref{E:fab2}, that $
(\mathbb{F}_{\bar{F}\bar{G}}(\alpha),\mathbb{F}_{\bar{F}\bar{G}}(\beta)) = (\mathbb{F}_{\bar{F}\bar{G}}(\alpha),\mathbb{F}_{\bar{F}\bar{G}}([\phi]_{\sim})) = (\mathbb{F}_{\bar{F}\bar{G}}(\alpha),[\phi]_{\sim}) \in \iota(E)$

\textit{Case 4}. $\mathbb{P} = R$. If  $\alpha, \beta, \gamma \in \mathbb{A}$ are such that $(\alpha, \beta, \gamma) \in \iota(R)$, then, arguing as in Case 2, we must have: (1) $\alpha, \gamma, \mathbb{F}_{\bar{F}\bar{G}}(\alpha), \mathbb{F}_{\bar{F}\bar{G}}(\gamma)\in Seq$, (2) $\beta \in \iota(S)$, in other words, $\beta = [\phi]_{\sim}$ for some $\phi \in \mathcal{L}$, (3) $\alpha\mathrel{\prec}\mathbb{F}_{\bar{F}\bar{G}}(\alpha)$, and $\gamma\mathrel{\prec}\mathbb{F}_{\bar{F}\bar{G}}(\gamma)$, and, finally, (4) for some $(\Xi,\Theta)\in W_c$ and some $\psi \in [\phi]_{\sim}$ we must have $\gamma = \alpha^\frown(\psi, (\Xi, \Theta))$. Now, Definition \ref{D:standard-sequence} implies that, for some $(\Xi',\Theta')\in W_c$ such that $(\Xi,\Theta)\leq_c(\Xi',\Theta')$ we must have $\mathbb{F}_{\bar{F}\bar{G}}(\gamma) = \mathbb{F}_{\bar{F}\bar{G}}(\alpha)^\frown(\psi, (\Xi', \Theta'))$. Therefore, by Definition \ref{D:canonical-sheaf}, we must have
$$
(\mathbb{F}_{\bar{F}\bar{G}}(\alpha), \mathbb{F}_{\bar{F}\bar{G}}(\beta), \mathbb{F}_{\bar{F}\bar{G}}(\gamma)) = (\mathbb{F}_{\bar{F}\bar{G}}(\alpha), \mathbb{F}_{\bar{F}\bar{G}}([\phi]_{\sim}), \mathbb{F}_{\bar{F}\bar{G}}(\gamma)) =  (\mathbb{F}_{\bar{F}\bar{G}}(\alpha), [\phi]_{\sim}, \mathbb{F}_{\bar{F}\bar{G}}(\gamma))\in \iota(R).
$$  
\end{proof}
Next, we show that the Kripke sheaf $\mathfrak{S}_c$ satisfies \eqref{E:th12}:
\begin{lemma}\label{L:indiscernibles}
	Let $\bar{F} \in Glob$. Then $\mathfrak{S}_c,\bar{F} \models_{fo}  \eqref{E:th12}$.
\end{lemma}
\begin{proof}
	Let $a, b \in \mathbb{A}$ be such that $
	\mathfrak{S}_c,\bar{F} \models_{fo} Sx\wedge Sz\wedge (\forall y)_O (Eyx\leftrightarrow Eyz)[x/a, z/b]$.
	Then $a,b \in \iota(S)$, that is to say, for some $\phi, \psi \in \mathcal{L}$, we must have $a = [\phi]_{\sim}$ and $b = [\psi]_{\sim}$. Assume, towards, contradiction, that $a = [\phi]_{\sim}\neq [\psi]_{\sim} = b$, then we must have $(\phi\leftrightarrow \psi)\notin \mathsf{IntCK}$; in view of Lemma \ref{L:truth}, we can suppose, wlog, that for some $(\Gamma, \Delta) \in W_c$ we have $\phi \in \Gamma$ but $\psi \notin \Gamma$. Since we clearly have $(\Gamma, \Delta) \in Seq \subseteq \mathbb{A}$, it follows from Definition \ref{D:canonical-sheaf} that we must have $
	\mathfrak{S}_c,\bar{F} \models_{fo} Sx\wedge Sz\wedge Oy\wedge Eyx[x/a, y/(\Gamma, \Delta), z/b]$,
	and, on the other hand, $\mathfrak{S}_c,\bar{F} \not\models_{fo} Eyz[y/(\Gamma, \Delta), z/b]$, which contradicts our assumption. Therefore, we must have $a = b$.
\end{proof} 
The next lemma can be seen as a version of `truth lemma' for $\mathfrak{S}_c$.
\begin{lemma}\label{L:sheaf-truth}
	Let $\bar{F} \in Glob$, $\alpha \in Seq$, $x \in Ind$, and let $\phi\in \mathcal{L}$. Then the following statements hold:
		\begin{enumerate}
			\item $\mathfrak{S}_c,\bar{F} \models_{fo} ST_x(\phi)[x/\alpha] \Leftrightarrow \phi \in \pi^1(end(\alpha))$.
			
			\item $\mathfrak{S}_c,\bar{F}\models_{fo} Sx\wedge (\forall y)_O(Eyx\leftrightarrow ST_y(\phi))[x/[\phi]_{\sim}]$.
		\end{enumerate}		
\end{lemma}
\begin{proof}
	We observe, first, that, for any given $\phi \in \mathcal{L}$, Part 1 clearly implies Part 2. Indeed, if Part 1 holds for a given $\phi$ and for all instantiations of $\bar{F}$, $\alpha$, and $x$, then we clearly must have $\mathfrak{S}_c,\bar{F}\models_{fo} Sx[x/[\phi]_{\sim}]$, and, as for the other conjunct, assume that a $\bar{G} \in Glob$ is such that $\bar{F}\mathrel{\sqsubseteq}\bar{G}$. Then, by Part 1, we must have, for every $\beta \in Seq$:
	$$
	E(\beta,[\phi]_{\sim}) \Leftrightarrow \phi \in \pi^1(end(\beta)) \Leftrightarrow \mathfrak{S}_c,\bar{G}\models_{fo} ST_x(\phi)[x/\beta].
	$$
	Since we have $[\phi]_{\sim} = \mathbb{F}_{\bar{F}\bar{G}}([\phi]_{\sim})$, we conclude that $
	\mathfrak{S}_c,\bar{F}\models_{fo} (\forall y)_O(Eyx\leftrightarrow ST_y(\phi)))[x/[\phi]_{\sim}]$.
	
	We will therefore prove both parts simultaneously by induction on the construction of $\phi\in \mathcal{L}$; but, in view of the foregoing observation, we will only argue for Part 1.
	
		\textit{Basis}. Assume that $\phi \in \mathcal{L}$ is atomic. The following cases are possible:
	
	\textit{Case 1}. $\phi = p \in Var$. Then $ST_x(\phi) = px$ and Definition \ref{D:canonical-sheaf} implies that we have:
	$$
	\mathfrak{S}_c,\bar{F}\models_{fo} px[x/\alpha] \Leftrightarrow \alpha \in \iota(p) \Leftrightarrow p \in \pi^1(end(\alpha)).
	$$
	
	Case 2 and Case 3, where we have $\phi = \top$ and $\phi = \bot$, respectively, are solved similarly.
	
	\textit{Induction step}. Again, several cases are possible:
	
	\textit{Case 1}. $\phi = \psi \wedge \chi$. Then $ST_x(\phi) = ST_x(\psi)\wedge ST_x(\chi)$ and we have, by IH and Lemma \ref{L:truth}:
	\begin{align*}
		\mathfrak{S}_c,\bar{F}\models_{fo} (ST_x(\psi)\wedge ST_x(\chi))[x/\alpha] &\Leftrightarrow \mathfrak{S}_c,\bar{F}\models_{fo} ST_x(\psi)[x/\alpha]\,\&\,\mathfrak{S}_c,\bar{F}\models_{fo} ST_x(\chi)[x/\alpha]\\
		&\Leftrightarrow \psi\in \pi^1(end(\alpha))\,\&\,\chi\in \pi^1(end(\alpha))\\
		 &\Leftrightarrow \mathcal{M}_c, end(\alpha)\models_c \psi\,\&\, \mathcal{M}_c, end(\alpha)\models_c \chi\\
		  &\Leftrightarrow \mathcal{M}_c, end(\alpha)\models_c \psi\wedge\chi\\
		 &\Leftrightarrow \psi\wedge\chi\in \pi^1(end(\alpha)) 
	\end{align*}

\textit{Case 2}. $\phi = \psi \vee \chi$. Similar to Case 1.

\textit{Case 3}. $\phi = \psi \to \chi$. Then $ST_x(\phi) = ST_x(\psi)\to ST_x(\chi)$, and, by  Lemma \ref{L:truth}, we know that:
\begin{align}
	\psi\to&\chi\in \pi^1(end(\alpha)) \Leftrightarrow \mathcal{M}_c, end(\alpha)\models_c \psi\to\chi\notag\\
	&\Leftrightarrow(\forall(\Xi,\Theta)\in W_c)(end(\alpha)\leq_c(\Xi,\Theta)\,\&\,\mathcal{M}_c, (\Xi,\Theta)\models_c \psi \Rightarrow \mathcal{M}_c, (\Xi,\Theta)\models_c \chi)\notag\\
	&\Leftrightarrow(\forall(\Xi,\Theta)\in W_c)(end(\alpha)\leq_c(\Xi,\Theta)\,\&\,\psi \in \Xi \Rightarrow \chi\in \Xi)\label{E:to}\tag{$\to$}
\end{align}
We now argue as follows:

$(\Leftarrow)$ Assume that $\psi\to\chi\in \pi^1(end(\alpha))$, and let
$\bar{G} \in Glob$ be such that $\bar{F}\mathrel{\sqsubseteq}\bar{G}$. Then, by Lemma \ref{L:technical}.2, $end(\alpha)\leq_c end(\mathbb{F}_{\bar{F}\bar{G}}(\alpha))$, so that \eqref{E:to} implies that $
 \psi \in \pi^1(end(\mathbb{F}_{\alpha\beta}(\alpha))) \Rightarrow \chi\in \pi^1(end(\mathbb{F}_{\bar{F}\bar{G}}(\alpha)))$,
 whence, by IH, it follows that
 $$
 \mathfrak{S}_c,\bar{G}\models_{fo} ST_x(\psi)[x/end(\mathbb{F}_{\bar{F}\bar{G}}(\alpha))] \Rightarrow \mathfrak{S}_c,\bar{G}\models_{fo} ST_x(\chi)[x/end(\mathbb{F}_{\bar{F}\bar{G}}(\alpha))].
 $$
 Since $\bar{G}\sqsupseteq\bar{F}$ was chosen in $Glob$ arbitrarily, we have shown that $\mathfrak{S}_c,\bar{F}\models_{fo} (ST_x(\psi)\to ST_x(\chi))[x/\alpha]$, or, equivalently, that $\mathfrak{S}_c,\bar{F}\models_{fo} ST_x(\phi)[x/\alpha]$.

$(\Rightarrow)$ We argue by contraposition. Assume that $\psi\to\chi\notin \pi^1(end(\alpha))$. By \eqref{E:to}, there must be a  $(\Xi,\Theta)\in W_c$ such that $end(\alpha)\leq_c(\Xi,\Theta)$, $\psi \in \Xi$, and $\chi\notin \Xi$. By Lemma \ref{L:global-functions}.2, we can choose a $\beta \in Seq$ such that $end(\beta) = (\Xi,\Theta)$, and a $F \in \mathfrak{G}$ such that $F(\alpha) = \beta$. But then we can set $\bar{G} := \bar{F}^\frown(F) \in Glob$; we clearly have $\bar{G}\sqsupseteq\bar{F}$, $\mathbb{F}_{\bar{F}\bar{G}}(\alpha) = F(\alpha) = \beta$, and $\pi^1(end(\beta)) = \Xi$. Therefore, under these settings, we also get that $\psi \in \pi^1(end(\mathbb{F}_{\bar{F}\bar{G}}(\alpha)))$, but $\chi\notin \pi^1(end(\mathbb{F}_{\bar{F}\bar{G}}(\alpha)))$. Thus we must have, by IH, that $\mathfrak{S}_c,\bar{G}\models_{fo} ST_x(\psi)[x/\mathbb{F}_{\bar{F}\bar{G}}(\alpha)]$ but  $\mathfrak{S}_c,\bar{G}\not\models_{fo} ST_x(\chi)[x/\mathbb{F}_{\bar{F}\bar{G}}(\alpha)]$. The latter means that $\mathfrak{S}_c,\bar{F}\not\models_{fo} (ST_x(\psi)\to ST_x(\chi))[x/\alpha]$, or, equivalently, that $\mathfrak{S}_c,\bar{F}\not\models_{fo} ST_x(\phi)[x/\alpha]$. 

\textit{Case 4}. $\phi = \psi\boxto\chi$. Then we have 
\begin{align*}
	ST_x(\phi) &:= \exists y(Sy \wedge (\forall z)_O(Ezy\leftrightarrow ST_z(\psi))\wedge\forall w(Rxyw\to ST_w(\chi))),
\end{align*}
and, by Lemma \ref{L:truth}, we know that:
\begin{align}
	\psi\boxto&\chi\in \pi^1(end(\alpha)) \Leftrightarrow \mathcal{M}_c, end(\alpha)\models_c \psi\boxto\chi\notag\\
	&\Leftrightarrow(\forall(\Xi,\Theta),(\Xi',\Theta')\in W_c)\notag\\
	&\qquad\quad(end(\alpha)\leq_c(\Xi,\Theta)\,\&\,R_c((\Xi,\Theta), \|\psi\|_{\mathcal{M}_c}, (\Xi',\Theta')) \Rightarrow \mathcal{M}_c, (\Xi',\Theta')\models_c \chi)\notag\\
	&\Leftrightarrow(\forall(\Xi,\Theta),(\Xi',\Theta')\in W_c)\notag\\
	&\qquad\quad(end(\alpha)\leq_c(\Xi,\Theta)\,\&\,R_c((\Xi,\Theta), \|\psi\|_{\mathcal{M}_c}, (\Xi',\Theta')) \Rightarrow \chi\in \Xi')\label{E:boxto}\tag{$\boxto$}
\end{align}
We now argue as follows:

$(\Leftarrow)$ Assume that $\psi\boxto\chi\in \pi^1(end(\alpha))$, let $\beta \in Seq$, and let
$\bar{G} \in Glob$ be such that $\bar{F}\mathrel{\sqsubseteq}\bar{G}$. If $(\mathbb{F}_{\bar{F}\bar{G}}(\alpha), [\psi]_{\sim},\beta) \in \iota(R)$, then we must have, by Definition \ref{D:canonical-sheaf}, that, for some $(\Gamma, \Delta)\in W_c$ and some $\theta\in [\psi]_{\sim}$, we have $\beta = \mathbb{F}_{\bar{F}\bar{G}}(\alpha)^\frown(\theta, (\Gamma, \Delta))$. But then, by Definition \ref{D:standard-sequence}, also $R_c(end(\mathbb{F}_{\bar{F}\bar{G}}(\alpha)), \|\theta\|_{\mathcal{M}_c}, (\Gamma, \Delta))$; and, since $\theta\in [\psi]_{\sim}$, Lemma \ref{L:truth} implies that $\|\theta\|_{\mathcal{M}_c} = \|\psi\|_{\mathcal{M}_c}$, so that $R_c(end(\mathbb{F}_{\bar{F}\bar{G}}(\alpha)), \|\psi\|_{\mathcal{M}_c}, (\Gamma, \Delta))$. By Lemma \ref{L:technical}.2, we must further have $end(\alpha)\leq_c end(\mathbb{F}_{\bar{F}\bar{G}}(\alpha))$, and now \eqref{E:boxto} yields that
 $\chi \in \Gamma = \pi^1(end(\beta))$. Therefore, it follows by IH that $\mathfrak{S}_c,\bar{G}\models_{fo} ST_w(\chi)[w/\beta]$.

Since the choice of $\beta \in Seq$ and 
$\bar{G} \in Glob$ such that $\bar{F}\mathrel{\sqsubseteq}\bar{G}$ was made arbitrarily, we have shown that $
\mathfrak{S}_c,\bar{F}\models_{fo} \forall w(Rxyw\to ST_w(\chi))[x/\alpha, y/[\psi]_{\sim}]$.
Moreover, IH for Part 2 implies that $
\mathfrak{S}_c,\bar{F}\models_{fo} Sy\wedge (\forall z)_O(Ezy\leftrightarrow ST_z(\psi)))[y/[\psi]_{\sim}]$. Summing up the two conjuncts and applying the existential generalization to $y$, we obtain that:
\begin{align*}
\mathfrak{S}_c,\bar{F}\models_{fo} \exists y(Sy \wedge (\forall z)_O(Ezy\leftrightarrow ST_z(\psi))\wedge\forall w(Rxyw\to ST_w(\chi)))[x/\alpha],
\end{align*}
or,in other words, that $\mathfrak{S}_c,\bar{F}\models_{fo} ST_x(\phi)[x/\alpha]$, as desired.

($\Rightarrow$) Arguing by contraposition, we assume that $\psi\boxto\chi\notin \pi^1(end(\alpha))$. In this case, \eqref{E:boxto} implies the existence of $(\Xi,\Theta),(\Xi',\Theta')\in W_c$ such that $end(\alpha)\leq_c(\Xi,\Theta)$, $R_c((\Xi,\Theta), \|\psi\|_{\mathcal{M}_c}, (\Xi',\Theta'))$, and $\chi\notin \Xi'$. By Lemma \ref{L:global-functions}.2, there exist a $\beta \in Seq(\Gamma, \Delta)$ such that $end(\beta) = (\Xi,\Theta)$ and an $F \in \mathfrak{G}$ such that $F(\alpha) = \beta$. But then clearly $\bar{G} = \bar{F}^\frown(F) \in Glob$ and $\bar{F}\mathrel{\sqsubseteq}\bar{G}$. Moreover, Definition \ref{D:canonical-sheaf} implies that $\mathbb{F}_{\bar{F}\bar{G}}(\alpha) = \beta$. Next, $R_c((\Xi,\Theta), \|\psi\|_{\mathcal{M}_c}, (\Xi',\Theta'))$ implies that $\gamma = \beta^\frown(\psi, (\Xi',\Theta')) \in Seq$, whence, by Definition \ref{D:canonical-sheaf}, $(\beta, [\psi]_\sim, \gamma)\in \iota(R)$. So, all in all we get that $
\mathfrak{S}_c,\bar{G}\models_{fo} Rxyw[x/\mathbb{F}_{\bar{F}\bar{G}}(\alpha), y/[\psi]_\sim, w/\gamma]$,
and, on the other hand, since $\chi\notin \Xi' = \pi^1(end(\gamma))$, IH for Part 1 implies that $
\mathfrak{S}_c,\bar{G}\not\models_{fo} ST_w(\chi)[w/\gamma]$.
Therefore, given that $\bar{F}\mathrel{\sqsubseteq}\bar{G}$, it follows that:
\begin{equation}\label{E:dagger}\tag{$\dagger$}
	\mathfrak{S}_c,\bar{F}\not\models_{fo} \forall w(Rxyw\to ST_w(\chi))[x/\alpha, y/[\psi]_\sim]
\end{equation}
If now $a \in \mathbb{A}$ is such that we have:
\begin{equation}\label{E:ddagger}\tag{$\ddagger$}
\mathfrak{S}_c,\bar{F}\models_{fo} Sy\wedge (\forall z)_O(Ezy\leftrightarrow ST_z(\psi))[y/a]	
\end{equation}
then note that, by IH for Part 2 we also have $
\mathfrak{S}_c,\bar{F}\models_{fo}Sy\wedge(\forall z)_O(Ezy\leftrightarrow ST_z(\psi))[y/[\psi]_\sim]$.
Therefore, Lemma \ref{L:indiscernibles} implies that we must have $a = [\psi]_\sim$ in this case. But then \eqref{E:dagger} implies that we must have
$\mathfrak{S}_c,\bar{F}\not\models_{fo} \forall w(Rxyw\to ST_w(\chi))[x/\alpha, y/a]$.
Since $a$ was chosen in $\mathbb{A}$ arbitrarily under the condition given by \eqref{E:ddagger}, we have shown that
\begin{align*}
	\mathfrak{S}_c,\bar{F}\not\models_{fo} \exists y(Sy \wedge (\forall z)_O(Ezy\leftrightarrow ST_z(\psi))\wedge\forall w(Rxyw\to ST_w(\chi)))[x/\alpha],
\end{align*}
or, in other words, that $\mathfrak{S}_c,\bar{F}\not\models_{fo} ST_x(\phi)[x/\alpha]$, as desired.

\textit{Case 5}. $\phi = \psi\diamondto\chi$. Then we have 
\begin{align*}
	ST_x(\phi) &:= \exists y(Sy \wedge (\forall z)_O(Ezy\leftrightarrow ST_z(\psi))\wedge\exists w(Rxyw\wedge ST_w(\chi))),
\end{align*}
And, by Lemma \ref{L:truth}, we know that:
\begin{align}
	\psi\diamondto&\chi\in \pi^1(end(\alpha)) \Leftrightarrow \mathcal{M}_c, end(\alpha)\models_c \psi\diamondto\chi\notag\\
	&\Leftrightarrow(\exists(\Xi,\Theta)\in W_c)(R_c(end(\alpha), \|\psi\|_{\mathcal{M}_c}, (\Xi,\Theta))\,\&\,\mathcal{M}_c, (\Xi,\Theta)\models_c \chi)\notag\\
	&\Leftrightarrow(\exists(\Xi,\Theta)\in W_c)(R_c(end(\alpha), \|\psi\|_{\mathcal{M}_c}, (\Xi,\Theta))\,\&\,\chi\in \Xi)\label{E:diamondto}\tag{$\diamondto$}
\end{align}
We now argue as follows:

$(\Leftarrow)$ Assume that $\psi\diamondto\chi\in \pi^1(end(\alpha))$. Then, by \eqref{E:diamondto}, we can choose a $(\Xi,\Theta)\in W_c$ such that both $R_c(end(\alpha), \|\psi\|_{\mathcal{M}_c}, (\Xi,\Theta))$ and $\chi\in \Xi$. The former means, by Definition \ref{D:standard-sequence}, that $\beta = \alpha^\frown(\psi, (\Xi,\Theta)) \in Seq$, and the latter means that $\chi \in \pi^1(end(\beta))$ whence, by IH for Part 1, it follows that $\mathfrak{S}_c,\bar{F}\models_{fo} ST_w(\chi)[w/\beta]$.

Finally, note that Definition \ref{D:canonical-sheaf} implies that $\mathfrak{S}_c,\bar{F}\models_{fo} Rxyw(\phi)[x/\alpha, y/[\psi]_\sim, w/\beta]$ and that, by IH for Part 2 we have $\mathfrak{S}_c,\bar{F}\models_{fo}Sy\wedge(\forall z)_O (Ezy\leftrightarrow ST_z(\psi))[y/[\psi]_\sim]$. 
Summing everything up, we have thus shown that
\begin{align*}
	\mathfrak{S}_c,\bar{F}\models_{fo} \exists y(Sy \wedge (\forall z)_O(Ezy\leftrightarrow ST_z(\psi))\wedge\exists w(Rxyw\wedge ST_w(\chi)))[x/\alpha],
\end{align*}
in other words, we have shown that $\mathfrak{S}_c,\bar{F}\models_{fo} ST_x(\phi)[x/\alpha]$, as desired.

($\Rightarrow$) Again, we argue by contraposition. Assume that $\psi\diamondto\chi\notin \pi^1(end(\alpha))$. In this case, \eqref{E:diamondto} implies that for every $(\Xi,\Theta)\in W_c$ such that $R_c(end(\alpha), \|\psi\|_{\mathcal{M}_c}, (\Xi,\Theta))$, we must have $\chi\notin \Xi$. Now, if $\beta \in Seq$ is such that $(\alpha, [\psi]_\sim, \beta) \in \iota(R)$, then, by Definition \ref{D:canonical-sheaf}, there must exist a $(\Gamma, \Delta) \in W_c$ and a $\theta \in [\psi]_\sim$ such that $\beta = \alpha^\frown(\theta, (\Gamma, \Delta))$. Since $\beta \in Seq$, Definition \ref{D:standard-sequence} implies that $R_c(end(\alpha), \|\theta\|_{\mathcal{M}_c}, (\Gamma, \Delta))$; but, since $\theta \in [\psi]_\sim$, the latter means that we must have $R_c(end(\alpha), \|\psi\|_{\mathcal{M}_c}, (\Gamma, \Delta))$. But then our assumption implies that $\chi\notin \Gamma = \pi^1(end(\beta))$, whence, by IH for Part 1, we must have $\mathfrak{S}_c,\bar{F}\not\models_{fo} ST_w(\chi)[w/\beta]$. Since $\beta$ was chosen in $Seq$ arbitrarily under the condition that $(\alpha, [\psi]_\sim, \beta) \in \iota(R)$, we have shown that:
\begin{equation}\label{E:natural}\tag{$\natural$}
	\mathfrak{S}_c,\bar{F}\not\models_{fo}\exists w(Rxyw \wedge ST_w(\chi))[x/\alpha, y/[\psi]_\sim]
\end{equation}
Next, suppose that $a \in \mathbb{A}$ is such that we have:
\begin{equation}\label{E:sharp}\tag{$\sharp$}
	\mathfrak{S}_c,\bar{F}\models_{fo} Sy\wedge (\forall z)_O(Ezy\leftrightarrow ST_z(\psi))[y/a]	
\end{equation}
then note that, by IH for Part 2 we have $
\mathfrak{S}_c,\bar{F}\models_{fo}Sy\wedge(\forall z)_O(Ezy\leftrightarrow ST_z(\psi))[y/[\psi]_\sim]$.
Therefore, Lemma \ref{L:indiscernibles} implies that we must have $a = [\psi]_\sim$ in this case. But then \eqref{E:natural} implies that $\mathfrak{S}_c,\bar{F}\not\models_{fo} \exists w(Rxyw \wedge ST_w(\chi))[x/\alpha, y/a]$.
Since $a$ was chosen in $\mathbb{A}$ arbitrarily under the condition given by \eqref{E:sharp}, we have shown that
\begin{align*}
	\mathfrak{S}_c,\bar{F}\not\models_{fo} \exists y(Sy \wedge (\forall z)_O(Ezy\leftrightarrow ST_z(\psi))\wedge\exists w(Rxyw \wedge ST_w(\chi)))[x/\alpha],
\end{align*}
or, in other words, that $\mathfrak{S}_c,\bar{F}\not\models_{fo} ST_x(\phi)[x/\alpha]$, as desired.
\end{proof}
Before we move on, we would like to draw a corollary from our lemma:
\begin{corollary}\label{C:simplified}
Let $\bar{F} \in Glob$, $\alpha \in Seq$, $x \in Ind$, and let $\psi,\chi\in \mathcal{L}$. Then the following statements hold:
\begin{enumerate}
	\item $\mathfrak{S}_c,\bar{F} \models_{fo} ST_x(\phi\boxto\psi)[x/\alpha] \Leftrightarrow \mathfrak{S}_c,\bar{F} \models_{fo} \forall w(Rxyw\to ST_w(\chi))[x/\alpha, y/[\psi]_\sim]$.
	
	\item $\mathfrak{S}_c,\bar{F} \models_{fo} ST_x(\phi\diamondto\psi)[x/\alpha] \Leftrightarrow \mathfrak{S}_c,\bar{F} \models_{fo} \exists w(Rxyw \wedge ST_w(\chi))[x/\alpha, y/[\psi]_\sim]$.
\end{enumerate}			
\end{corollary}
\begin{proof}
	(Part 1) ($\Leftarrow$) Trivial by Lemma \ref{L:sheaf-truth}.2.
	
	($\Rightarrow$) If $\mathfrak{S}_c,\bar{F} \models_{fo} ST_x(\phi\boxto\psi)[x/\alpha]$, then we can choose an $a \in \mathbb{A}$ such that:
	\begin{align*}
		\mathfrak{S}_c,\bar{F}\models_{fo} Sy \wedge (\forall z)_O(Ezy\leftrightarrow ST_z(\psi))\wedge\forall w(Rxyw\to ST_w(\chi))[x/\alpha, y/a]
	\end{align*}
	But then, by Lemmas \ref{L:indiscernibles} and \ref{L:sheaf-truth}.2, we must have $a = [\psi]_\sim$, so that $\mathfrak{S}_c,\bar{F} \models_{fo} \forall w(R(x,y,w)\to ST_w(\chi))[x/\alpha, y/[\psi]_\sim]$ follows. Part 2 is proved similarly to Part 1.	
\end{proof}
It only remains now to show that $\mathfrak{S}_c$ is a model of $Th$:
\begin{lemma}\label{L:th}
	For every $\bar{F} \in Glob$, we have $\mathfrak{S}_c, \bar{F}\models_{fo} Th$.
\end{lemma}
\begin{proof}
	That the Kripke sheaf $\mathfrak{S}_c$ must satisfy \eqref{E:th1}--\eqref{E:th5} is clear from Definition \ref{D:canonical-sheaf}. The satisfaction of \eqref{E:th6}--\eqref{E:th8} is implied by Lemma \ref{L:sheaf-truth}.2 (one needs to instantiate $y$ by $[p]_\sim$, $[\top]_\sim$, or $[\bot]_\sim$, respectively). The satisfaction of \eqref{E:th12} was shown in Lemma \ref{L:indiscernibles}. We consider the remaining parts of $Th$ in more detail below:
	
	\eqref{E:th9}. Let $\ast \in \{\wedge, \vee, \to\}$, $\bar{F} \in Glob$, and $a, b \in \iota(S)$. Then, by Definition \ref{D:canonical-sheaf}, there must exist some $\phi, \psi \in \mathcal{L}$ such that $a = [\phi]_\sim$ and $b = [\psi]_\sim$. But then Lemma \ref{L:sheaf-truth}.2 implies that we have all of the following:
	\begin{align}
		\mathfrak{S}_c, \bar{F}&\models_{fo} Sx \wedge (\forall w)_O(Ewx\leftrightarrow ST_w(\phi))[x/a]\label{E:phi}\tag{$\S$}\\
		\mathfrak{S}_c, \bar{F}&\models_{fo} Sy \wedge (\forall w)_O(Ewy\leftrightarrow ST_w(\psi))[y/b]\label{E:psi}\tag{$\P$}\\
		\mathfrak{S}_c, \bar{F}&\models_{fo} Sz \wedge (\forall w)_O(Ewz\leftrightarrow ST_w(\phi)\ast ST_w(\psi))[z/[\phi\ast\psi]_\sim]\notag
	\end{align}
	whence $\mathfrak{S}_c, \bar{F}\models_{fo} \eqref{E:th9}$ clearly follows.
	
	\eqref{E:th10}. Again, let $\bar{F} \in Glob$ and $a, b \in \iota(S)$. Then let $\phi, \psi \in \mathcal{L}$ be such that $a = [\phi]_\sim$ and $b = [\psi]_\sim$. Lemma \ref{L:sheaf-truth}.2 implies that both \eqref{E:phi} and \eqref{E:psi} hold, and that we have:
	$$
	\mathfrak{S}_c, \bar{F}\models_{fo} Sz \wedge (\forall w)_O(Ewz\leftrightarrow ST_w(\phi\boxto\psi))[z/[\phi\boxto\psi]_\sim].
	$$
	By Corollary \ref{C:simplified}, it follows now that we must have:
	$$
	\mathfrak{S}_c, \bar{F}\models_{fo} Sz \wedge (\forall w)_O(Ewz\leftrightarrow \forall w'(Rwxw'\to ST_{w'}(\chi)))[x/a,z/[\phi\boxto\psi]_\sim].
	$$
	whence $\mathfrak{S}_c, \bar{F}\models_{fo} \eqref{E:th10}$ clearly follows. The case of \eqref{E:th11} is parallel to the case of \eqref{E:th10}.
\end{proof}
We are now finally in a position to prove a converse to Corollary \ref{C:easy}:
\begin{proposition}\label{P:hard}
	For all $\Gamma, \Delta \subseteq \mathcal{L}$ and for every $x \in Ind$, if $Th, \{ST_x(\phi)\mid \phi \in \Gamma\}\models_{fo}\{ST_x(\psi)\mid \psi \in \Delta\}$, then  $\Gamma \models_{\mathsf{IntCK}} \Delta$.
\end{proposition}
\begin{proof}
	We argue by contraposition. If $\Gamma \not\models_{\mathsf{IntCK}} \Delta$, then $(\Gamma, \Delta)$ must be satisfiable, and therefore, by Lemma \ref{L:lindenbaum}, we can choose a $(\Gamma', \Delta')\in W_c$ such that $(\Gamma', \Delta')\supseteq (\Gamma, \Delta)$. By Definition \ref{D:standard-sequence}, we have $(\Gamma', \Delta')\in Seq$, therefore, Lemma \ref{L:sheaf-truth}.1 and Lemma \ref{L:th} together imply that:
	$$
	\mathfrak{S}_c, \Lambda \models_{fo} (Th \cup \{ST_x(\phi)\mid \phi \in \Gamma\}, \{ST_x(\psi)\mid \psi \in \Delta\})[x/(\Gamma', \Delta')],
	$$ 
	or, equivalently, that $Th, \{ST_x(\phi)\mid \phi \in \Gamma\}\not\models_{fo}\{ST_x(\psi)\mid \psi \in \Delta\}$, as desired.
\end{proof}
We can now formulate and prove the main result of this subsection:
\begin{theorem}\label{T:foil}
	For all $\Gamma, \Delta \subseteq \mathcal{L}$ and for every $x \in Ind$, we have $\Gamma \models_{\mathsf{IntCK}} \Delta$ iff $Th, \{ST_x(\phi)\mid \phi \in \Gamma\}\models_{fo}\{ST_x(\psi)\mid \psi \in \Delta\}$.
\end{theorem}
\begin{proof}
	By Corollary \ref{C:easy} and Proposition \ref{P:hard}.
\end{proof}

\section{Conclusion, discussion, and future work}\label{S:conclusion}
We have shown that $\mathsf{IntCK}$ is indeed the correct version of basic intuitionistic conditional logic in the sense outlined in the opening paragraphs of this paper. Thus, Theorem \ref{T:completeness} shows that $\mathsf{IntCK}$ is basic in the sense that it is strongly complete relative to a (suitably defined) universal class of Kripke models; Theorem \ref{T:foil} then shows that $\mathsf{IntCK}$ is intuitionistic in the sense that it is strongly complete relative to an intuitionistic reading of the classical semantics of conditional logic. Finally, $\mathsf{IntCK}$ is fully conditional in that it features the full set of conditional connectives $\{\boxto, \diamondto\}$ which are not definable in terms of one another.

It seems that the construction of $\mathfrak{S}_c$ used in the proof of Theorem \ref{T:foil} is relatively novel, since similar results for intuitionistic modal logic are proved by other methods; in particular, \cite{simpson} proceeds proof-theoretically whereas \cite{ewald} uses the method of selective filtration forming a countable chain of finite models. The latter method was not very convenient to use in the case of conditional logic since one has to keep a supply of counterexamples distinguishing modal accessibility relations induced by formulas that fail to be provable complete.

Our answer to the question of what is the basic intuitionistic conditional logic is still open to criticism, mainly in relation to the intuitionistic component of our claims. We would like to briefly mention here two possible counter-arguments. First, despite the fact that Theorem \ref{T:foil} shows that the reasoning given in $\mathsf{IntCK}$ is but a subsystem of the first-order intuitionistic reasoning and can be embedded into the latter by the same sort of a standard translation that is also appropriate in the classical case, the fact that our proof of Theorem \ref{T:foil} is itself decidedly classical, diminishes the foundational importance of this result in the eyes of an intuitionist. Secondly, the theory $Th$ used in this result is open to doubts as to whether it smuggles in too much of a classical set-theoretic principles to be acceptable for an intuitionist.

As for the second concern, we note that \eqref{E:th1}--\eqref{E:th5} are clearly harmless principles typical for two-sorted formulations of $\mathsf{FOIL}$, and \eqref{E:th12} is a form of extensionality axiom; the latter is generally uncontroversial and present in every known form of constructive set theory. Finally, \eqref{E:th6}--\eqref{E:th11} are particular forms of comprehension. Even though the question about the intuitionistically acceptable amount of comprehension is definitely open and contested, the comprehension principles given in \eqref{E:th6}--\eqref{E:th11} all seem to be very tame in that they only use the formulas with guarded quantifiers over the object sort. It seems reasonable to expect, therefore, that they will be acceptable under any of the existing accounts of intuitionistc set theory.

As for the first concern, however, we can only acknowledge it as a drawback of our work; to do better in this respect, one should rather prove Theorem \ref{T:foil} in the spirit of \cite[Ch. 5]{simpson}, and we hope that we will be able to publish in the near future some sort of continuation to the present paper in which we will close this gap.

Another major direction for future work is to extend the methods and results of this paper to the treatment of conditionals in constructive logics with strong negation, for example, to Nelson's logics $\mathsf{N3}$ and $\mathsf{N4}$, and to the negation-inconsistent connexive logic $\mathsf{C}$ introduced by H. Wansing in \cite{w}. Among these three systems, $\mathsf{C}$ looks, perhaps, the most promising one, given that this subject already has seen its first rather intriguing steps in \cite{wu}, and the methods of the current paper seem to open a way to a considerable refinement of these first results.

\textbf{Acknowledgements}. This research has received funding from the European Research Council (ERC) under the European Union's Horizon 2020 research and innovation programme, grant agreement ERC-2020-ADG, 101018280, ConLog.

\appendix

\section{Proof of Lemma \ref{L:theorems}}\label{A:1}
We sketch the respective proofs and derivations in $\mathbb{ICK}$:
\begin{align}
	\eqref{E:Rnec}:\qquad\phi\label{E:p1}& &&\text{premise}\\
	\phi& \leftrightarrow \top\label{E:p3}&&\text{\eqref{E:p1}, \eqref{E:a0}, \eqref{E:mp}}\\
	(\psi&\boxto\phi)\leftrightarrow(\psi\boxto\top)\label{E:p4}&&\text{\eqref{E:p3}, \eqref{E:RCbox}}\\
	\psi&\boxto\phi\label{E:p5}&&\text{\eqref{E:p4}, \eqref{E:a6},  \eqref{E:a0}, \eqref{E:mp}}
\end{align}
\begin{align}
	\eqref{E:Rmbox}:\qquad\phi&\to\psi\label{E:p6} &&\text{premise}\\
	(\phi&\wedge\psi)\leftrightarrow \phi\label{E:p7}&&\text{\eqref{E:p6},\eqref{E:a0}, \eqref{E:mp}}\\
	(\chi&\boxto(\phi\wedge\psi))\leftrightarrow(\chi\boxto\phi)\label{E:p8}&&\text{\eqref{E:p7}, \eqref{E:RCbox}}\\
	(\chi&\boxto\phi)\to(\chi\boxto\psi)\label{E:p10}&&\text{\eqref{E:p8},  \eqref{E:a1}, \eqref{E:a0}, \eqref{E:mp}}
\end{align}
\begin{align}
	\eqref{E:Rmdiam}:\quad\phi&\to\psi\label{E:q0} &&\text{premise}\\
	(\phi&\vee\psi)\leftrightarrow \psi\label{E:q1}&&\text{\eqref{E:q0},\eqref{E:a0}, \eqref{E:mp}}\\
	(\chi&\diamondto(\phi\vee\psi))\leftrightarrow(\chi\diamondto\psi)\label{E:q2}&&\text{\eqref{E:q1}, \eqref{E:RCdiam}}\\
	(\chi&\diamondto\phi)\to(\chi\diamondto\psi)&&\text{\eqref{E:q2}, \eqref{E:a3}, \eqref{E:a0}, \eqref{E:mp}}
\end{align}
\begin{align}
	\eqref{E:T1}:\qquad(\psi&\wedge(\psi\to\chi))\to\chi\label{E:q5}&&\text{\eqref{E:a0}, \eqref{E:mp}}\\
	(\phi&\boxto(\psi\wedge(\psi\to\chi)))\to(\phi\boxto\chi)\label{E:q6}&&\text{\eqref{E:q5}, \eqref{E:Rmbox}}\\
	((\phi&\boxto\psi)\wedge(\phi\boxto(\psi\to\chi)))\to(\phi\boxto\chi)&&\text{\eqref{E:a1}, \eqref{E:q6},\eqref{E:a0}, \eqref{E:mp}}
\end{align}
\begin{align}
\eqref{E:T2}:\,	((\phi&\diamondto\psi)\wedge(\phi\boxto(\psi\to\chi)))\to(\phi\diamondto(\psi\wedge(\psi\to\chi)))\label{E:q7} &&\text{\eqref{E:a2}}\\
	(\psi&\wedge(\psi\to\chi))\to\chi\label{E:q8}&&\text{\eqref{E:a0}, \eqref{E:mp}}\\
	(\phi&\diamondto(\psi\wedge(\psi\to\chi)))\to(\phi\diamondto\chi)\label{E:q9}&&\text{\eqref{E:q8}, \eqref{E:Rmdiam}}\\
	((\phi&\diamondto\psi)\wedge(\phi\boxto(\psi\to\chi)))\to(\phi\diamondto\chi)&&\text{\eqref{E:q7}, \eqref{E:q9},\eqref{E:a0}, \eqref{E:mp}}
\end{align}
\begin{align}
	\eqref{E:T3}:\quad\psi&\to((\psi\to\chi)\to\chi)\label{E:r0}&&\text{\eqref{E:a0}, \eqref{E:mp}}\\
	(\phi&\boxto\psi)\to(\phi\boxto((\psi\to\chi)\to\chi))\label{E:r1}&&\text{\eqref{E:r0}, \eqref{E:Rmbox}}\\
	(\phi&\boxto\psi)\to((\phi\diamondto(\psi\to\chi))\to(\phi\diamondto\chi))&&\text{\eqref{E:r1},\eqref{E:T2},\eqref{E:a0}, \eqref{E:mp}}
\end{align}	
\begin{align}
	\eqref{E:T4}:\quad(\phi&\boxto(\psi \to \bot))\to((\phi\diamondto\psi)\to(\phi\diamondto\bot))\label{E:r5}&&\text{\eqref{E:T2}}\\
	((\phi&\diamondto\psi)\to(\phi\diamondto\bot))\to((\phi\diamondto\psi)\to\bot)\label{E:r7}&&\text{\eqref{E:a6},\eqref{E:a0}, \eqref{E:mp}}\\
	(\phi&\boxto\neg\psi)\to\neg(\phi\diamondto\psi)&&\text{\eqref{E:r5},\eqref{E:r7},\eqref{E:a0}, \eqref{E:mp}}\\
	((\phi&\diamondto\psi)\to \bot)\to((\phi\diamondto\psi)\to(\phi\boxto\bot))\label{E:r9}&&\text{\eqref{E:a0}, \eqref{E:mp}}\\
	\neg&(\phi\diamondto\psi)\to(\phi\boxto\neg\psi)&&\text{\eqref{E:r9},\eqref{E:a4},\eqref{E:a0}, \eqref{E:mp}}
\end{align}

\section{Proof of Lemma \ref{L:CK}}\label{A:2}
	(Part 1) Note that \eqref{E:a3} follows from \eqref{E:a1} in $\mathsf{CK}$ by applying \eqref{E:ax1} plus duality principles; \eqref{E:a6} can be deduced from \eqref{E:a5} similarly. Moreover, \eqref{E:Rmbox} and \eqref{E:T1} can be deduced in $\mathsf{CK}$ in the same way as in $\mathsf{IntCK}$. We sketch the proofs for the remaining axioms and inference rules:
	\begin{align}
		\eqref{E:a2}:\quad&(\phi\boxto\chi)\to(\phi\boxto(\neg(\psi\wedge \chi)\to\neg\psi))\label{E:ck2}&&\text{\eqref{E:a0},\eqref{E:mp},\eqref{E:Rmbox}}\\
		&(\phi\boxto\chi)\to((\phi\boxto\neg(\psi\wedge \chi))\to(\phi\boxto\neg\psi))\label{E:ck3}&&\text{\eqref{E:ck2},\eqref{E:T1},\eqref{E:a0},\eqref{E:mp}}\\
		&(\phi\boxto\chi)\to(\neg(\phi\boxto\neg\psi)\to\neg(\phi\boxto\neg(\psi\wedge \chi)))\label{E:ck4}&&\text{\eqref{E:ck3},\eqref{E:a0},\eqref{E:mp}}\\
		&(\phi\boxto\chi)\to((\phi\diamondto\psi)\to(\phi\diamondto(\psi\wedge \chi)))&&\text{\eqref{E:ck4},\eqref{E:ax1},\eqref{E:a0},\eqref{E:mp}}
	\end{align} 
	\begin{align}
		\eqref{E:a4}:\quad&(\phi\diamondto\psi)\to(\phi\boxto\chi)\label{E:ck5}&&\text{premise}\\
		&\neg(\phi\boxto\neg\psi)\to(\phi\boxto\chi)\label{E:ck10}&&\text{\eqref{E:ck5}, \eqref{E:ax1},\eqref{E:a0},\eqref{E:mp}}\\
		&(\phi\boxto\neg\psi)\to(\phi\boxto(\psi\to\chi))\label{E:ck8}&&\text{\eqref{E:a0},\eqref{E:mp},\eqref{E:Rmbox}}\\
		&(\phi\boxto\chi)\to(\phi\boxto(\psi\to\chi))\label{E:ck9}&&\text{\eqref{E:a0},\eqref{E:mp},\eqref{E:Rmbox}}\\
		&\neg(\phi\boxto\neg\psi)\to(\phi\boxto(\psi\to\chi))\label{E:ck11}&&\text{\eqref{E:ck10},\eqref{E:ck9},\eqref{E:a0},\eqref{E:mp}}\\
		&((\phi\boxto\neg\psi)\vee\neg(\phi\boxto\neg\psi))\to(\phi\boxto(\psi\to\chi))\label{E:ck12}&&\text{\eqref{E:ck8},\eqref{E:ck11},\eqref{E:a0},\eqref{E:mp}}\\
		&(\phi\boxto\neg\psi)\vee\neg(\phi\boxto\neg\psi)\label{E:ck13}&&\text{\eqref{E:ax0}}\\
		&\phi\boxto(\psi\to\chi)\label{E:ck14}&&\text{\eqref{E:ck12},\eqref{E:ck13},\eqref{E:mp}}
	\end{align}
	\begin{align}
		\eqref{E:RAdiam}:\quad&\phi\leftrightarrow\psi \in \mathsf{CK}\label{E:kc1}&&\text{premise}\\
		&(\phi\boxto\neg\chi)\leftrightarrow(\psi\boxto\neg\chi) \in \mathsf{CK}\label{E:kc2}&&\text{\eqref{E:kc1},\eqref{E:RAbox}}\\
		&\neg(\phi\boxto\neg\chi)\leftrightarrow\neg(\psi\boxto\neg\chi) \in \mathsf{CK}\label{E:kc3}&&\text{\eqref{E:kc2},\eqref{E:a0},\eqref{E:mp}}\\
		&(\phi\diamondto\chi)\leftrightarrow(\psi\diamondto\chi) \in \mathsf{CK}&&\text{\eqref{E:kc3},\eqref{E:ax1},\eqref{E:a0},\eqref{E:mp}}
	\end{align} 
	\begin{align}
		\eqref{E:RCdiam}:\quad&\phi\leftrightarrow\psi \in \mathsf{CK}\label{E:kc4}&&\text{premise}\\
		&(\chi\boxto\neg\phi)\leftrightarrow(\chi\boxto\neg\psi)\in \mathsf{CK}\label{E:kc6}&&\text{\eqref{E:kc4},\eqref{E:a0},\eqref{E:mp},\eqref{E:RCbox}}\\
		&(\chi\diamondto\phi)\leftrightarrow(\chi\diamondto\psi)\in \mathsf{CK}&&\text{\eqref{E:kc6},\eqref{E:ax1},\eqref{E:a0},\eqref{E:mp}}
	\end{align}
	Having now every element of $\mathbb{ICK}$ deduced in $\mathsf{CK}$, we can deduce the remaining parts of Lemma \ref{L:theorems} as it was done in Section \ref{sub:axiomatization}.
	
	As for Part 2, note that \eqref{E:ax0} intuitionistically implies $\neg\neg(\phi\diamondto\psi)\leftrightarrow(\phi\diamondto\psi)$, whence  \eqref{E:ax1} follows by \eqref{E:T4}.

\section{Proofs of some technical results from Section \ref{sub:foil}}\label{A:3}
All of the sketches in this Appendix are semi-formal but allow for an easy completion into full proofs in any complete Hilbert-style axiomatization of $\mathsf{FOIL}$. The following lemma lists most of the intuitionistic principles assumed in this appendix:
\begin{lemma}\label{L:foil-assumptions}
	Let $\Gamma \cup \{\phi, \psi\} \subseteq \mathcal{L}_{fo}$, and let $x \in Ind$.  
	Then all of the following statements hold:
	\begin{align}
		\Gamma, \phi \models_{fo} \psi &\Leftrightarrow \Gamma\models_{fo} \phi \to \psi\label{E:DT}\tag{DT}\\
		\Gamma \models_{fo} \phi \to \psi  &\Rightarrow \Gamma \models_{fo} \exists x\phi \to \psi\label{E:Bern}\tag{Bern} &&x\notin FV(\Gamma\cup \{\psi\})\\
		\Gamma \models_{fo} \phi &\Rightarrow \Gamma \models_{fo} \forall x\phi\label{E:Gen}\tag{Gen} &&x\notin FV(\Gamma)
	\end{align} 
\end{lemma}
In case our formulas get too long, we will be replacing then with their labels, writing e.g. $(123) \to (124)$ instead of $\phi \to \psi$ in case $\phi$ did occur earlier as equation $(123)$ and $\psi$ as equation $(124)$.
\begin{proof}[Proof of Lemma \ref{L:th-existence-1}]
	(Part 1) Consider the following deduction D1 from premises:
	\begin{align}
		&Sy' \wedge (\forall z)_O(Ezy\leftrightarrow ST_z(\psi))\label{E:cof1}&&\text{premise}\\
		&\forall w(Rxy'w\to ST_w(\chi))\label{E:cof2}&&\text{premise}\\
		&Sy \wedge (\forall z)_O(Ezy\leftrightarrow ST_z(\psi))\label{E:cof3}&&\text{premise}\\
		&(\forall z)_O(Ezy\leftrightarrow Ezy')\label{E:cof4}&&\text{by \eqref{E:cof1}, \eqref{E:cof3}}\\
		&y \equiv y'\label{E:cof5}&&\text{by \eqref{E:cof4}, \eqref{E:th12}}\\
		&\forall w(Rxyw\to ST_w(\chi))\label{E:cof6}&&\text{by \eqref{E:cof2}, \eqref{E:cof5}}
	\end{align}
We now reason as follows:
\begin{align*}
	&Th, \eqref{E:cof1}, \eqref{E:cof2}\models_{fo} \forall y(Sy \wedge (\forall z)_O(Ezy\leftrightarrow ST_z(\psi))\to\forall w(Rxyw\to ST_w(\chi)))&&\text{(D1, \eqref{E:DT}, \eqref{E:Gen})}\\
	&Th\models_{fo} \exists y'(\eqref{E:cof1}\wedge\eqref{E:cof2})\to \forall y(Sy \wedge (\forall z)_O(Ezy\leftrightarrow ST_z(\psi))\to\forall w(Rxyw\to ST_w(\chi)))&&\text{(\eqref{E:DT}, \eqref{E:Bern})}
\end{align*}
	Now the definition of $ST$ yields the result claimed for Part 1. Part 2 is proved by a parallel argument.
\end{proof}
\begin{proof}[Proof of Lemma \ref{L:th-existence-2}]
	We proceed by induction on the construction of $\phi\in \mathcal{L}$.
	
	\textit{Basis}. If $\phi = p \in Var$ (resp. $p = \bot, \top$), then the Lemma follows by \eqref{E:th6} (resp. \eqref{E:th7}, \eqref{E:th8}).
	
	\textit{Induction step}. The following cases are possible:
	
	\textit{Case 1}. $\phi = \psi\ast\chi$, where $\ast\in \{\wedge, \vee, \to\}$. We consider first-order deduction D2:
	\begin{align}
		&Sx\wedge(\forall w)_O(Ewx\leftrightarrow ST_w(\psi))\label{E:fo1}&&\text{premise}\\
		&Sy\wedge(\forall w)_O(Ewy\leftrightarrow ST_w(\chi))\label{E:fo2}&&\text{premise}\\
		&Sz\wedge(\forall w)_O(Ewz\leftrightarrow (Ewx\ast Ewy))\label{E:fo3}&&\text{premise}\\
		&(\forall w)_O((Ewx\ast Ewy)\leftrightarrow(ST_w(\psi)\ast ST_w(\chi)))\label{E:fo4}&&\text{by \eqref{E:fo1}, \eqref{E:fo2}}\\
		&(\forall w)_O(Ewz\leftrightarrow(ST_w(\psi)\ast ST_w(\chi)))\label{E:fo5}&&\text{by \eqref{E:fo3}, \eqref{E:fo4}}\\
		&\exists z(Sz\wedge (\forall w)_O(Ewz\leftrightarrow ST_w(\psi\ast\chi)))\label{E:fo7}&&\text{by \eqref{E:fo3}, \eqref{E:fo5}, def. of $ST$}
	\end{align}
We now reason as follows:
\begin{align*}
	&Th, \eqref{E:fo1}, \eqref{E:fo2} \models_{fo} \exists z\eqref{E:fo3}\to \exists z(Sz\wedge (\forall w)_O(Ewz\leftrightarrow ST_w(\psi\ast\chi)))&&\text{(D2, \eqref{E:DT}, \eqref{E:Bern})}\\
	&Th, \eqref{E:fo1}, \eqref{E:fo2} \models_{fo} Sx\wedge Sy &&\text{(trivially)}\\
	&Th,  \eqref{E:fo1}, \eqref{E:fo2} \models_{fo} (Sx\wedge Sy )\to \exists z\eqref{E:fo3}&&\text{\eqref{E:th9}}\\
	&Th, \eqref{E:fo1}, \eqref{E:fo2} \models_{fo}\exists z(Sz\wedge (\forall w)_O(Ewz\leftrightarrow ST_w(\psi\ast\chi)))&&\text{\eqref{E:mp}}\\
	&Th\models_{fo} \exists x\eqref{E:fo1} \to (\exists z\eqref{E:fo3}\to \exists z(Sz\wedge (\forall w)_O(Ewz\leftrightarrow ST_w(\psi\ast\chi))))&&\text{\eqref{E:DT}, \eqref{E:Bern}}\\
	&Th\models_{fo} \exists x\eqref{E:fo1} \wedge \exists y\eqref{E:fo2}&&\text{(IH)}\\
	&Th\models_{fo}\exists z(Sz\wedge (\forall w)_O(Ewz\leftrightarrow ST_w(\psi\ast\chi)))&&\text{\eqref{E:mp}} 
\end{align*}

	\textit{Case 2}. $\phi = \psi\boxto\chi$. We consider the following deductions from premises. 
	
	Deduction D3:
	\begin{align}
		&Sy'\wedge(\forall w)_O(Ewy'\leftrightarrow ST_w(\psi))\label{E:ffo1}&&\text{premise}\\
		&\forall w(Rxy'w\to ST_w(\chi))\label{E:ffo2}&&\text{premise}\\
		&\exists y((\eqref{E:ffo1}\wedge\eqref{E:ffo2})^y_{y'})\label{E:ffo4}&&\text{by \eqref{E:ffo1},  \eqref{E:ffo2}}\\
		&ST_x(\psi\boxto\chi)\label{E:ffo5}&&\text{by \eqref{E:ffo4}, def. of $ST$}
	\end{align}
Deduction D4:
\begin{align}
	&\eqref{E:ffo1}&&\text{premise}\notag\\
	&Sy\wedge(\forall w)_O(Ewy\leftrightarrow ST_w(\psi))\label{E:ffo7}&&\text{premise}\\	
	&\forall w(Rxyw\to ST_w(\chi))\label{E:ffo8}&&\text{premise}\\
	&(\forall w)_O(Ewy\leftrightarrow Ewy)\label{E:ffo8a}&&\text{by \eqref{E:ffo1}, \eqref{E:ffo7}}\\
	&y \equiv y'\label{E:ffo9}&&\text{by \eqref{E:ffo8a}, \eqref{E:th12}}\\
	&\forall w(Rxy'w\to ST_w(\chi))&&\text{by \eqref{E:ffo8}, \eqref{E:ffo9}}
\end{align}  
	D3 and D4 lead to the following intermediary results:
\begin{align}
	Th, \eqref{E:ffo1}\models_{fo}\forall w(Rxy'w\to ST_w(\chi))\to ST_x(\psi\boxto\chi)\label{fo1} &&\text{(D3, \eqref{E:DT})}\\
	Th, \eqref{E:ffo1}\models_{fo}\exists y(\eqref{E:ffo7}\wedge\eqref{E:ffo8})\to\forall w(Rxy'w\to ST_w(\chi))\label{fo1a}&&\text{(D4, \eqref{E:DT}, \eqref{E:Bern})}\\
	Th, \eqref{E:ffo1}\models_{fo}ST_x(\psi\boxto\chi)\to\forall w(Rxy'w\to ST_w(\chi))\label{fo2}&&\text{(\eqref{fo1a}, def. of $ST$)}\\
	Th, \eqref{E:ffo1}\models_{fo}\forall x(ST_x(\psi\boxto\chi)\leftrightarrow\forall w(Rxy'w\to ST_w(\chi)))\label{fo3}&&\text{\eqref{fo1},\eqref{fo2}, \eqref{E:Gen}}
\end{align}	
	We now feed these results into the next deduction D5:
	\begin{align}
		&\eqref{E:ffo1}&&\text{premise}\notag\\
		&Sz\wedge(\forall w)_O(Ewz\leftrightarrow ST_w(\chi))\label{E:ffo13}&&\text{premise}\\	
		&Sy'\wedge Sz\label{E:ffo14}&&\text{by \eqref{E:ffo1}, \eqref{E:ffo13}}\\
		&\exists z'(\forall x)_O(Exz'\leftrightarrow \forall w(Rxy'w\to Ewz))\label{E:ffo15}&&\text{by \eqref{E:ffo14}, \eqref{E:th10}}\\
		&\exists z'(\forall x)_O(Exz'\leftrightarrow \forall w(Rxy'w\to ST_w(\chi)))\label{E:ffo16}&&\text{by \eqref{E:ffo13}, \eqref{E:ffo15}}\\
		&\exists z'(\forall x)_O(Exz'\leftrightarrow ST_x(\psi\boxto\chi))\label{E:ffo17}&&\text{by \eqref{E:ffo16}, \eqref{fo3}}
	\end{align}
	We now finish our reasoning as follows:
	\begin{align*}
		Th\models_{fo}\exists y'\eqref{E:ffo1} \to (\exists z\eqref{E:ffo13} \to \exists z'(\forall x)_O(Exz'\leftrightarrow ST_x(\psi\boxto\chi)))&&\text{(D5, \eqref{E:DT}, \eqref{E:Bern})}\\
		Th\models_{fo} \exists y'\eqref{E:ffo1} \wedge \exists z\eqref{E:ffo13}&&\text{(IH)}\\
		Th\models_{fo}\exists z'(\forall x)_O(Exz'\leftrightarrow ST_x(\psi\boxto\chi)))&&\eqref{E:mp}
	\end{align*}

	\textit{Case 3}. $\phi = \psi\diamondto\chi$. Parallel to Case 2.
\end{proof} 
\begin{proof}[Proof of Proposition \ref{P:easy}]
	We show that the standard translation of every axiom of $\mathbb{ICK}$ first-order-follows from $Th$ and that the rules of $\mathbb{ICK}$ preserve this property. First, note that every instance of \eqref{E:a0} is translated  into an instance of \eqref{E:a0} and hence a $\mathsf{FOIL}$-valid formula; the same is true for every instance of \eqref{E:a6}. By Lemma \ref{L:th-existence-2}, we also know that every instance of \eqref{E:a5} first-order-follows from $Th$. Next, the applications of \eqref{E:mp}  translate into applications of this same rule \eqref{E:mp}, and the applications of every rule in the set $\{\eqref{E:RAbox}, \eqref{E:RCbox}, \eqref{E:RAdiam}, \eqref{E:RCdiam}\}$ translate to applications of some $\mathsf{FOIL}$-deducible rule. It remains to consider the instances of axiomatic schemas \eqref{E:a1}--\eqref{E:a4}, which is a tedious but straightforward exercise in first-order intuitionistic reasoning. We display the reasoning for \eqref{E:a4} as an example.
	
	Let $\phi,\psi,\chi \in \mathcal{L}$. Consider the following first-order deduction D6 from premises:
	\begin{align}
		&ST_x(\phi\diamondto\psi)\to ST_x(\phi\boxto\chi)\label{E:ea1}&&\text{premise}\\
		&Sy\wedge (\forall z)_O(Ezy\leftrightarrow ST_z(\phi))\label{E:ea2}&&\text{premise}\\
		&Rxyw\wedge ST_w(\psi)\label{E:ea3}&&\text{premise}\\
		&\exists w(Rxyw\wedge ST_w(\psi))\label{E:ea4}&&\text{by \eqref{E:ea3}}\\
		&\exists y(\eqref{E:ea2}\wedge\eqref{E:ea4})\label{E:ea4a}&&\text{by \eqref{E:ea2}, \eqref{E:ea4}}\\
		&ST_x(\phi\diamondto\psi)\label{E:ea5}&&\text{by \eqref{E:ea4a}, def. of $ST$}\\
		&ST_x(\phi\boxto\chi)\label{E:ea6}&&\text{by \eqref{E:ea1}, \eqref{E:ea5}}\\
		&\forall y(Sy \wedge (\forall z)_O(Ezy\leftrightarrow ST_z(\psi)))\to\notag\\
		&\qquad\qquad\qquad\qquad\to\forall w(Rxyw\to ST_w(\chi)))\label{E:ea7}&&\text{by \eqref{E:ea6}, Lemma \ref{L:th-existence-1}}\\
		&\forall w(Rxyw\to ST_w(\chi))\label{E:ea8}&&\text{by \eqref{E:ea2}, \eqref{E:ea7}}\\
		&ST_w(\chi)\label{E:ea9}&&\text{by \eqref{E:ea3}, \eqref{E:ea8}}
	\end{align}
We now reason as follows:
\begin{align*}
	&Th, \eqref{E:ea1}, \eqref{E:ea2}\models_{fo} \forall w(Rxyw \to (ST_w(\psi) \to ST_w(\chi)))&&\text{(D5, \eqref{E:DT}, \eqref{E:Gen})}\\
	&Th, \eqref{E:ea1}, \eqref{E:ea2}\models_{fo} \forall w(R(x,y,w) \to ST_w(\psi\to\chi))&&\text{(def. of $ST$)}\\
	&Th, \eqref{E:ea1}, \eqref{E:ea2}\models_{fo} \exists y(\eqref{E:ea2}\wedge\forall w(Rxyw \to ST_w(\psi\to\chi)))\\
	&Th, \eqref{E:ea1}, \eqref{E:ea2}\models_{fo} ST_x(\phi\boxto(\psi\to\chi))&&\text{(def. of $ST$)}\\
	&Th, \eqref{E:ea1}\models_{fo} \exists y\eqref{E:ea2}\to ST_x(\phi\boxto(\psi\to\chi))&&\eqref{E:DT}, \eqref{E:Bern}\\
	&Th\models_{fo} \exists y\eqref{E:ea2}&&\text{(Lemma \ref{L:th-existence-2})}\\
	&Th, \eqref{E:ea1}\models_{fo}ST_x(\phi\boxto(\psi\to\chi))&&\eqref{E:mp}\\
	&Th\models_{fo}\eqref{E:ea1}\to ST_x(\phi\boxto(\psi\to\chi))&&\eqref{E:DT}\\
	&Th\models_{fo}ST_x(((\phi\diamondto\psi)\to (\phi\boxto\chi))\to(\phi\boxto(\psi\to\chi)))&&\text{(def. of $ST$)}
\end{align*}
\end{proof}	
}
\end{document}